\documentclass[10pt]{amsart}

\usepackage{amsmath} \usepackage{amssymb, amscd,cancel, graphicx,soul}
	\usepackage{pinlabel,mathtools}

\usepackage{hyperref}
\hypersetup{
    colorlinks=true, %set true if you want colored links
    linktoc=all,     %set to all if you want both sections and subsections linked
    linkcolor=black,  %choose some color if you want links to stand out
}

\usepackage{enumerate}

\usepackage{cleveref}

\usepackage{tikz-cd}

\usepackage{mathdots}
\newtheorem{thm}{Theorem}[section]
\newtheorem{letteredthm}{Theorem}

\newtheorem{letteredcor}[letteredthm]{Corollary}

\newtheorem{cor}[thm]{Corollary}
\newtheorem*{cor*}{Corollary}
\newtheorem{lem}[thm]{Lemma}
\newtheorem{prop}[thm]{Proposition}

\newtheorem*{ques*}{Question}
\newtheorem{defin}[thm]{Definition}

\newtheorem*{example*}{Example}
\newtheorem*{porism*}{Porism}
\newtheorem*{scholium*}{Scholium}

\newtheorem*{thm*}{Theorem}
\newtheorem*{defin*}{Definition}
\newtheorem*{lem*}{Lemma}
\newtheorem*{prop*}{Proposition}
\newtheorem*{remark*}{Remark}

\def\d{{\mathrm d}}
\def\cA{{\mathcal A}}

\def\cM{{\mathcal M}}

\def\JF{{\rm JF}}
\def\JFC{{\rm JFC}}

\def\pid{{\pi_{\rm d}}}

\def\wt{\widetilde}

\def\bC{{\mathbb C}}

\def\bF{{\mathbb F}}
\def\cG{{\mathcal G}}

\def\rad{{\mathrm{Rad}}}

\def\bR{{\mathbb R}}

\def\area{{\mathrm{Area}}}

\def\del{{\partial}}

\newcommand{\ol}{\overline}

\begin{document}
\thispagestyle{empty}
\title[Square pegs between two graphs]{Square pegs between two graphs}
\author{Joshua Evan Greene} 
\address{Department of Mathematics, Boston College, USA}
\email{joshua.greene@bc.edu}
\urladdr{https://sites.google.com/bc.edu/joshua-e-greene}
\author{Andrew Lobb} 
\address{Mathematical Sciences,
	Durham University,
	UK}
\email{andrew.lobb@durham.ac.uk}
\urladdr{http://www.maths.dur.ac.uk/users/andrew.lobb/}
\thanks{JEG was supported by the National Science Foundation under Award No.~DMS-2304856 and by a Simons Fellowship.}

\begin{abstract}
We show that there always exists an inscribed square in a Jordan curve given as the union of two graphs of functions of Lipschitz constant less than $1 + \sqrt{2}$.  We are motivated by Tao's result that there exists such a square in the case of Lipschitz constant less than $1$.  In the case of Lipschitz constant $1$, we show that the Jordan curve inscribes rectangles of every similarity class.  Our approach involves analysing the change in the spectral invariants of the Jordan Floer homology under perturbations of the Jordan curve.
\end{abstract}

\maketitle

\section{Introduction.}
\label{sec:intro}
A Jordan curve $\gamma \subset \bC$ is said to \emph{inscribe} a quadrilateral $Q$ if $\gamma$ contains four points forming the vertex set of such a $Q$.  Similarly $Q$ is said to \emph{inscribe in} $\gamma$.  The Square Peg Problem due to Toeplitz posits that every Jordan curve inscribes a square~\cite{toeplitz1911}.  This conjecture remains unsolved.

In 2017, Tao proved a result partly responsible for reviving current interest in the Square Peg Problem and its relatives.	Suppose that $f, g \colon [0,1] \rightarrow \bR$ are two functions satisfying $f(0) = g(0)$, $f(1) = g(1)$, and $f(t) > g(t)$ for all $0<t<1$.  Then the union of the graphs of $f$ and $g$ forms a Jordan curve $\gamma(f,g)$.

\begin{thm*}[Tao \cite{tao}]
	\label{thm:tao}
	If $f$ and $g$ as above are both Lipschitz continuous of Lipschitz constant less than $1$, then $\gamma(f,g)$ inscribes a square.
\end{thm*}

\noindent
In this paper we improve on this result.
A {\em $\theta$-rectangle} is a rectangle whose diagonals meet at angle $\theta$.

\begin{letteredthm}
	\label{thm:thetarect}
	If $0 < \theta \leq \pi/2$, and $f$ and $g$ are of Lipschitz constant less than $\tan(\frac{\theta + \pi}{4})$, then $\gamma(f,g)$ inscribes a $\theta$-rectangle.
	\end{letteredthm}

\noindent
As $\tan(\frac{\cdot + \pi}{4}) : [0,\pi/2] \longrightarrow [1,1+\sqrt{2}]$ is an increasing function, we obtain the following pair of corollaries:

\begin{letteredcor}
	\label{cor:square}
	If $f$ and $g$ are of Lipschitz constant less than $1 + \sqrt{2}$, then $\gamma(f,g)$ inscribes a square.
\end{letteredcor}

\begin{letteredcor}
	\label{cor:rect}
	If $f$ and $g$ are of Lipschitz constant $1$, then $\gamma(f,g)$ inscribes a $\theta$-rectangle, $\forall \theta \in (0,\pi/2]$.
\end{letteredcor}

\Cref{thm:thetarect} follows from properties of the Jordan Floer homology $\JF(\gamma,\theta)$ of a smooth Jordan curve $\gamma$ and angle $0 < \theta < \pi$.
For a generic pair $(\gamma,\theta)$, it is the homology of a complex generated over $\bF_2$ by the set $\cG$ of inscriptions of a $\theta$-rectangle in $\gamma$.
In earlier work, we defined this invariant in the case of a real analytic Jordan curve $\gamma$, establishing its basic properties and extracting some consequences for inscription problems \cite{greenelobb3}.
The main work in this paper is to extend the definition to a smooth Jordan curve and to establish some new properties for its associated spectral invariants, leading to \Cref{thm:thetarect} as an application.
The remainder of the introduction summarizes the approach and the new ingredients.

\newpage

Prior to \cite{greenelobb3}, we had proven that every smooth Jordan curve inscribes a $\theta$-rectangle for all $0 < \theta \leq \pi/2$ \cite{greenelobb1}.
The principal difficulty in extending this result to a weaker regularity class  of curves is that of {\em shrink-out}.
If $\gamma_n \rightarrow \gamma$ is a sequence of smooth Jordan curves converging to a non-smooth limit, then by passing to a subsequence, one finds a convergent sequence of inscribed $\theta$-rectangles $Q_n \subset \gamma_n$ with limit $Q_n \rightarrow Q \subset \gamma$.  In good cases, one concludes that $Q$ is a $\theta$-rectangle.
The problem that might arise is that  the limit $Q$ could \emph{degenerate} to a single point of $\gamma$.
In this work and in our previous work \cite{greenelobb3}, we use \emph{spectral invariants} in Jordan Floer homology in order to preclude shrink-out.

Jordan Floer homology is a version of Lagrangian Floer homology.
From this point of view, the generating set $\cG$ is the set of transverse intersection points between a pair of Lagrangian tori in $\bC^2$ associated with a generic pair $(\gamma,\theta)$.
Floer homology furnishes a real-valued grading on the generators called the \emph{action} $\cA \colon \cG \longrightarrow \bR$, and $\cA(g)$ has something to do with the size of the corresponding rectangle $Q_g \subset \gamma$.
In particular, if $0 < \cA(g) < \area(\gamma)$ (the area enclosed by $\gamma$) and we know an \emph{a priori} upper bound on the length of $\gamma$, then we can deduce a lower bound on the size of $Q_g$.

The differential of the chain complex $\JFC(\gamma, \theta)$ is filtered by action, leading to a filtration grading on the homology $\JF(\gamma, \theta)$.
One of the main results of \cite{greenelobb3} is that the group $\JF(\gamma,\theta)$ is a $2$-dimensional $\bF_2$-vector space generated by homology classes $\alpha_i$ for $i=1,2$ where $\alpha_i$ is homogeneous of homological degree $i$.  We write $\ell_i(\gamma, \theta) \in \bR$ for the filtration grading of $\alpha_i$: this is the \emph{spectral invariant} associated to $\alpha_i$.  For convenience we also define $\ell_i(\gamma, 0) = 0$ and $\ell_i(\gamma, \pi) = \area(\gamma)$.

Spectral invariants possess several important properties, which we collect here.

\begin{thm}[Properties of spectral invariants]
	\label{thm:basic_properties}
	Suppose that $\gamma$ is a smooth Jordan curve and $0 \leq \theta \leq \pi$.  Then the spectral invariants $\ell_1, \ell_2$ satisfy the following:
	\begin{enumerate}
		\item \emph{Area bound.} $0 \leq \ell_1(\gamma, \theta), \ell_2(\gamma, \theta) \leq \area(\gamma)$.
		\item \emph{Spectrality.}  Each $\ell_i(\gamma, \theta)$ is equal to the action of some generator of $\JFC(\gamma,\theta)$ corresponding to an inscribed $\theta$-rectangle of $\gamma$.
		% omitted the "oriented diagonal of"
		\item \emph{Monotonicity.}  Each $\ell_i(\gamma,\cdot)\colon [0,\pi] \longrightarrow [0,\area(\gamma)]$ is monotonic increasing.
		\item \emph{Continuity in $\theta$.}  Each $\ell_i(\gamma,\cdot)\colon [0,\pi] \longrightarrow [0,\area(\gamma)]$ is continuous.
		\item \emph{Continuity in $\gamma$.}  If $\gamma_r$, $r \in [0,1]$, is a smooth family of smooth Jordan curves, then $r \mapsto \ell_i(\gamma_r,\theta)$ is continuous in $r$.
		\item \emph{Hofer distance bound.} If $\gamma$ and $\gamma'$ enclose equal area, then $|\ell_i(\gamma,\theta) - \ell_i(\gamma',\theta)| \le 4 \d(\gamma,\gamma'),$ where $\d$ denotes the Hofer distance (\Cref{defin:hofernorm}).
		\item \emph{Duality.}  $\ell_1(\gamma, \theta) + \ell_2(\gamma, \pi - \theta) = \area(\gamma)$.
	\end{enumerate}
\end{thm}

Properties (1)-(4) already appeared in \cite{greenelobb3} (compare \cite[Theorem 3]{leczap}).
They were stated there for real analytic curves, but the properties carry over directly to smooth curves, following the extension of the construction of $\JFC(\gamma,\theta)$ to smooth Jordan curves carried out in \Cref{subsec:smoothing}.
These properties already suffice for some some interesting applications \cite[Theorem A]{greenelobb3}.

Properties (4) and (5) are closely related.
In both cases, continuity depends on understanding filtered chain homotopy equivalences between the Jordan Floer chain complexes attached to different pairs of data.
There are two established mechanisms for constructing these maps.
The original one, due to Floer, is the \emph{bifurcation} method \cite{floerlag}.
Floer used it to show the invariance of Lagrangian Floer homology under Hamiltonian perturbation.
The second one is the \emph{continuation} method expounded, for instance, in \cite{audindamian,aurouxintro}.
We established Property (4) in \cite{greenelobb3} using the continuation method.
This method presents analytic difficulties in attempting to prove Property (5).
Instead, we employ the bifurcation method to prove it in \Cref{prop:cts_in_deformation} below.
Its success depends in part on the extending the definition of the Jordan Floer chain complex from real analytic to smooth curves.

Property (6) appears as \Cref{cor:hoferdistancebound3} below, following a familiar line of argument using Hofer distance bounds on action.
Property (7) appears as \Cref{prop:duality} below.
It follows from a duality result relating $\JF(\gamma, \theta)$ to $\JF(\gamma, \pi - \theta)$.
The duality property is not actually required in the sequel, but we record it for posterity.

In previous work \cite{greenelobb3}, we showed that $\area(\gamma)/2 - \vert \area(\gamma)/2 - \ell_2(\gamma, \theta) \vert$ gives rise to a lower bound on the size of an inscribed $\theta$-rectangle of $\gamma$ (also depending on $\area(\gamma)$ and on the length of $\gamma$).  We then considered how quickly $\ell_2(\gamma, \cdot)$ can change with $\theta$, finding a relationship to the diameter $2\rad(\gamma)$.  By approximating a rectifiable (in other words, finite length) Jordan curve $\gamma$ by a sequence of real analytic Jordan curves of uniformly bounded length, we showed, for example, that $\gamma$ inscribes a square if $\area(\gamma)$ is at least half the area bounded by a circle of the same diameter as $\gamma$.

To derive results by varying now $\gamma$ instead of $\theta$, we would like to know that small perturbations in $\gamma$ give rise to small perturbations in the actions of generators of $\JFC(\gamma, \theta)$ -- and in fact to \emph{small enough} perturbations that we can use this knowledge to prove something useful.  This is the reason for our restriction to Jordan curves given as the union of two graphs of functions satisfying a certain Lipschitz condition.  Such Jordan curves inscribe only \emph{elegant} $\theta$-rectangles.

\begin{defin}
	\label{defin:maximally_graceful}
	(Elegance.)
	Suppose that $\Gamma \subset \bC$ is a rectangular Jordan curve whose vertex $Q$ set lies on the Jordan curve $\gamma \subset \bC$.  If $\gamma$ is continuously isotopic to $\Gamma$ through an isotopy fixing $Q$, then we say that $Q$ is an \emph{elegant} inscribed rectangle of $\gamma$.
\end{defin}

\noindent
Previously, Schwartz studied {\em gracefully} inscribed rectangles, those inscribed rectangles whose cyclic ordering of the vertices coming from travelling around the curve is a cyclic ordering when travelling around the rectangle \cite{schwartz}.  Elegant rectangles are necessarily graceful, and Schwartz's nomenclature inspired our own.

The generators of $\JFC(\gamma, \theta)$ that correspond to elegant rectangles have good behavior under small perturbations of $\gamma$, and this allows us to conclude Theorem~\ref{thm:thetarect}.  This good behavior is reflected below in \Cref{prop:varying_through_graceful}.

It is interesting to contrast Tao's argument \cite{tao} that $\gamma(f,g)$ inscribes a square when $f$ and $g$ are of Lipschitz constant less than $1$.  The Lipschitz condition was used in that argument to force any inscribed square to have two vertices on the graph of $f$ and two vertices on the graph of $g$.  We impose the weaker Lipschitz condition to guarantee that all inscribed rectangles are elegant.

We shall find it convenient, at some point, to work with piecewise linear (PL) Jordan curves rather than smooth Jordan curves.
The technical reason is that we want to approximate $\gamma(f,g)$ by Jordan curves for which we can simultaneously control the spectral invariants; for which all inscriptions of $\theta$-rectangles are elegant; and which are suitably generic.  The precise advantage we gain from working with PL Jordan curves is highlighted in a footnote to the proof of the following proposition in Section \ref{sec:lipschitz_et_cetera}.
Leaving for later the definition of PL isotopy and the spectral invariants associated to a PL Jordan curve, we show the following:

\begin{prop}
\label{prop:varying_through_graceful}
	Suppose that $\gamma_t \subset \bC$ for $0 \leq t \leq 1$ is a PL isotopy of PL Jordan curves such that $\gamma_s$ is contained inside the bounded complementary region to $\gamma_t$ for all $0 \leq s < t \leq 1$.  Further suppose that for each $0 \leq t \leq 1$, all $\theta$-rectangles inscribed in $\gamma_t$ are elegant.
	
	Then for $i=1,2$ we have that
	\[ \ell_i(\gamma_0, \theta) \leq \ell_i(\gamma_1, \theta) \leq  \ell_i(\gamma_0, \theta) + \area(\gamma_1) - \area(\gamma_0)  {\rm .} \]
\end{prop}

\noindent
Thus, spectral invariants of $\gamma_0$ bound those of $\gamma_1$.
We expect that \Cref{prop:varying_through_graceful} remains true if one replaces the hypotheses of piecewise linearity with hypotheses of smoothness.

The proof of Theorem \ref{thm:thetarect} then follows by approximating $\gamma(f,g)$ by a sequence of piecewise linear Jordan curves $\gamma_n$ and obtaining bounds on each $\ell_i(\gamma_n,\theta)$ by applying Proposition \ref{prop:varying_through_graceful}.

\subsection*{Plan of the paper.}

Section \ref{sec:JF} contains some technical work in the Jordan Floer complex -- extending the definition from real analytic to smooth Jordan curves, and verifying a duality result relating $\JF(\gamma, \theta)$ to $\JF(\gamma, \pi - \theta)$.

In Section \ref{sec:deformation} we work in more specificity.  Firstly, we apply Floer's bifurcation arguments to show that the spectral invariants vary continuously with isotopies of the Jordan curve.  In particular, we obtain bounds on the variation of the spectral invariants in terms of the Hofer distance between two Lagrangians.  Secondly, we extend the definition of the spectral invariants to piecewise linear Jordan curves and prove continuity of the spectral invariants in variations of piecewise linear curves.  Thirdly, we give more precise bounds on the rate of change of the spectral invariants when the Jordan curves vary through a family only inscribing elegant rectangles.

Finally in Section \ref{sec:lipschitz_et_cetera} we consider the situation of the Jordan curves $\gamma(f,g)$ of Theorem \ref{thm:thetarect}.

\subsection*{Acknowledgements.}
We thank Joseph Cima and Vsevolod Shevchishin for entertaining our questions about conformal maps, and we extend very special thanks to Steve Bell for providing us with an argument for Lemma~\ref{lem:hopf}.

Part of this research was carried out in Fall 2022 while the authors were visiting the Simons Laufer Mathematical Sciences Institute, which is supported by the National Science Foundation (Grant No. DMS-1928930).  We also thank the R\'enyi Institute and the Erd\H{o}s Center for their hospitality in Summer 2024.
\newpage

\section{Jordan Floer homology.}
\label{sec:JF}
In this section we carry out some technical work in the Jordan Floer chain complex.  In Subsection \ref{subsec:smoothing}, we verify that the definition of the Jordan Floer homology given in \cite{greenelobb3} for real analytic Jordan curves in fact extends to smooth Jordan curves.  Real analyticity was required in \cite{greenelobb3} so that the Schwarz reflection principle would apply, allowing us to bypass some analysis of boundary behavior of holomorphic maps with smooth boundary.  For smooth curves, we obtain the desired boundary behavior by a different method which involves a simple, elegant argument in complex analysis due to Steve Bell (\Cref{lem:hopf}).
 In Subsection \ref{subsec:dualityofJF}, we deduce a duality result for Jordan Floer homology which allows us conclude the duality result in \Cref{thm:basic_properties}.
\subsection{The smooth Jordan Floer complex.}
\label{subsec:smoothing}
Fix an angle $0 < \theta < \pi$ and a smooth Jordan curve $\gamma$.
We begin by summarizing the construction of the Jordan Floer chain complex $\JFC(\gamma, \theta)$ and where it was important that we required the real analyticity of $\gamma$.

Firstly, we have the Lagrangian $\gamma \times \gamma \subset \bC^2$.  The construction of the differential in the Jordan Floer chain complex $\JFC(\gamma, \theta)$ (as well as the package of continuation maps, chain homotopies, and so on) involves considering moduli spaces of strips of finite energy
\[ \Sigma = \bR \times [0,1] \longrightarrow \bC \]
that satisfy a Cauchy-Riemann-Floer equation with respect to some (possibly strip-dependent) Hamiltonian, and that map the boundary $\partial \Sigma$ to $\gamma \times \gamma$.  Crucially, we make the further requirement that the closure of the image of these strips should be disjoint from the diagonal
\[ \Delta_\bC = \{ (z,z) \in \bC^2 : z \in \bC \} \subset \bC^2 {\rm .} \]
The intersection of $\Delta_\bC$ with $\gamma \times \gamma$ is clean, and is a smooth curve
\[ \Delta_\gamma = \Delta_\bC \cap (\gamma \times \gamma) = \{ (z,z) \in \bC^2 : z \in \gamma \subset \bC \} {\rm .} \]

The argument that Floer homology is well-defined involves an analysis of moduli spaces of strips in which the moduli space has positive dimension.  It is therefore important for our definition that a $1$-parameter family of such strips
\[ u_r \colon \Sigma \longrightarrow \bC^2 \]
for $0 \leq r < 1$ cannot have as its limit a strip $u_1$ (or a degeneration of a strip) that intersects the diagonal $\Delta_\bC$.

Such unfortunate behavior may \emph{a priori} occur in several different ways: firstly the strips could break or bubble at $r=1$ at a point in $\Delta_\gamma$.  If there is no breaking or bubbling then at $r=1$ we have an honest limit strip $u_1$.  But perhaps this strip $u_1$ could intersect $\Delta_\bC$ in its interior, or the strip $u_1$ could be disjoint from $\Delta_\bC$ in its interior but could meet $\Delta_\gamma$ in its boundary.  It was mainly to preclude this last possibility, \emph{boundary touching}, that we enforced the real analyticity of $\gamma$.  We also used real analyticity to exclude bubbling at $\Delta_\gamma$ by a very similar argument.

We briefly rehash the argument that rules out boundary touching in the real analytic case, obscuring some details in the interests of brevity.

Suppose that $u_1(0,0) \in \Delta_\gamma$.  We choose a small neighborhood $H \subset \Sigma = \bR \times [0,1]$ of $(0,0)$, biholomorphic to the half-disc.
If $H$ is small enough, then $u_1$ satisfies the Cauchy-Riemann equation on $H$, since the Hamiltonians have been chosen to 0 near the boundaries of the strip and the almost-complex structures have been chosen to be standard near the diagonal $\Delta_\bC$.
Assuming that $\gamma$ is real analytic, Schwarz reflection allows us to extend the domain of definition of $u_1 \vert_H$ to a disc neighborhood of $(0,0) \subset \bR^2 = \bC$.
If $u_1(0,0)$ is a non-isolated intersection point of $u_1(\Sigma)$ with the diagonal $\Delta_\bC$, then analytic continuation implies that $u_1(\Sigma) \subset \bC$, contradicting what we know about the ends of the strip.  On the other hand, if $u_1(0,0)$ is isolated, then Schwarz reflection gives an isolated intersection point of two complex curves at that point,
and this cannot be perturbed away due to positivity of intersection.  Then Schwarz reflection near $u_1(0,0)$ of $u_r$ for $r$ close to $1$ shows that $u_r(\Sigma)$ was not disjoint from $\Delta_\bC$, contradicting our assumption.

Now that we understand the issues, we establish a proposition which rules out boundary touching without the assumption that $\gamma$ is real analytic; and, following that, a proposition that rules out bubbling at $\Delta_\gamma$.

\begin{prop}[No boundary touching]
	\label{prop:boundary_touching}
	Suppose that for $0 \leq r \leq 1$
	\[ u_r \colon (-1,1) \times [0,1) \rightarrow \bC^2 \]
	is a family of smooth maps
	such that
	$u_r((-1,1)\times\{0\}) \subset \gamma \times \gamma$, and
	such that each $u_r$ is holomorphic on its interior.
	Further suppose that the image of $u_r$ is disjoint from $\Delta_\bC$ for $0 \leq r < 1$, and that $u_1((-1,1)\times(0,1))$ is disjoint from $\Delta_\bC$.
	
	Then the image of $u_1$ is disjoint from $\Delta_\gamma$.
\end{prop}

\begin{proof}
	Let us suppose for a contradiction, without loss of generality, that $u_1(0,0) = (p,p) \in \Delta_\gamma$ and that $T_p\gamma = e^{\frac{i\pi}{4}}\bR \subset \bC$.
	
	We shall be making use of the projection
	\[ \pid \colon \bC^2 \longrightarrow \bC \colon (z,w) \longmapsto z-w \]
	(where the subscript refers to the \emph{direction}
	of the arc in $\bC$ with endpoints $z$ and $w$).  We note that $\pid^{-1}(0) = \Delta_\bC$, so that $(\pid \circ u_1)^{-1}(0) \subset (-1,1) \times \{ 0 \}$.
	
	Now $\pid \circ u_1$ is a non-constant map, holomorphic in its interior, so that the closed set $(\pid \circ u_1)^{-1}(0) \subset (-1,1) \times \{ 0 \}$ cannot contain any open interval or else $\pid \circ u_1$ will vanish identically \cite{chernoff}.  So let us pick $\epsilon > 0$ such that, without loss of generality\footnote{The maps $u_r$ do not meet the diagonal and therefore we can make the assumption that the arguments lie in the given interval and, in particular, not possibly in the interval $[-7\pi/8,-5\pi/8]$ as well.}
	\[ \pi/8 \leq \arg(\pid \circ u_1([-\epsilon, \epsilon] \times \{ 0 \})) \leq 3\pi/8 {\rm ,} \]
	and we have $\pid \circ u_1 (-\epsilon,0) \not= 0$ and $\pid \circ u_1 (\epsilon,0) \not= 0$.

	\begin{figure}
		\labellist
		%\scriptsize
		\pinlabel {$C$} at 350 790
		\pinlabel {$(-1,1) \times [0,1)$} at 500 1000
		\pinlabel {$(\epsilon,0)$} at 750 370
		\pinlabel {$(-\epsilon,0)$} at 250 370
		\pinlabel {$\pid \circ u_1 (C)$} at 2000 -70
		\endlabellist
		\centering
		\includegraphics[scale=0.12]{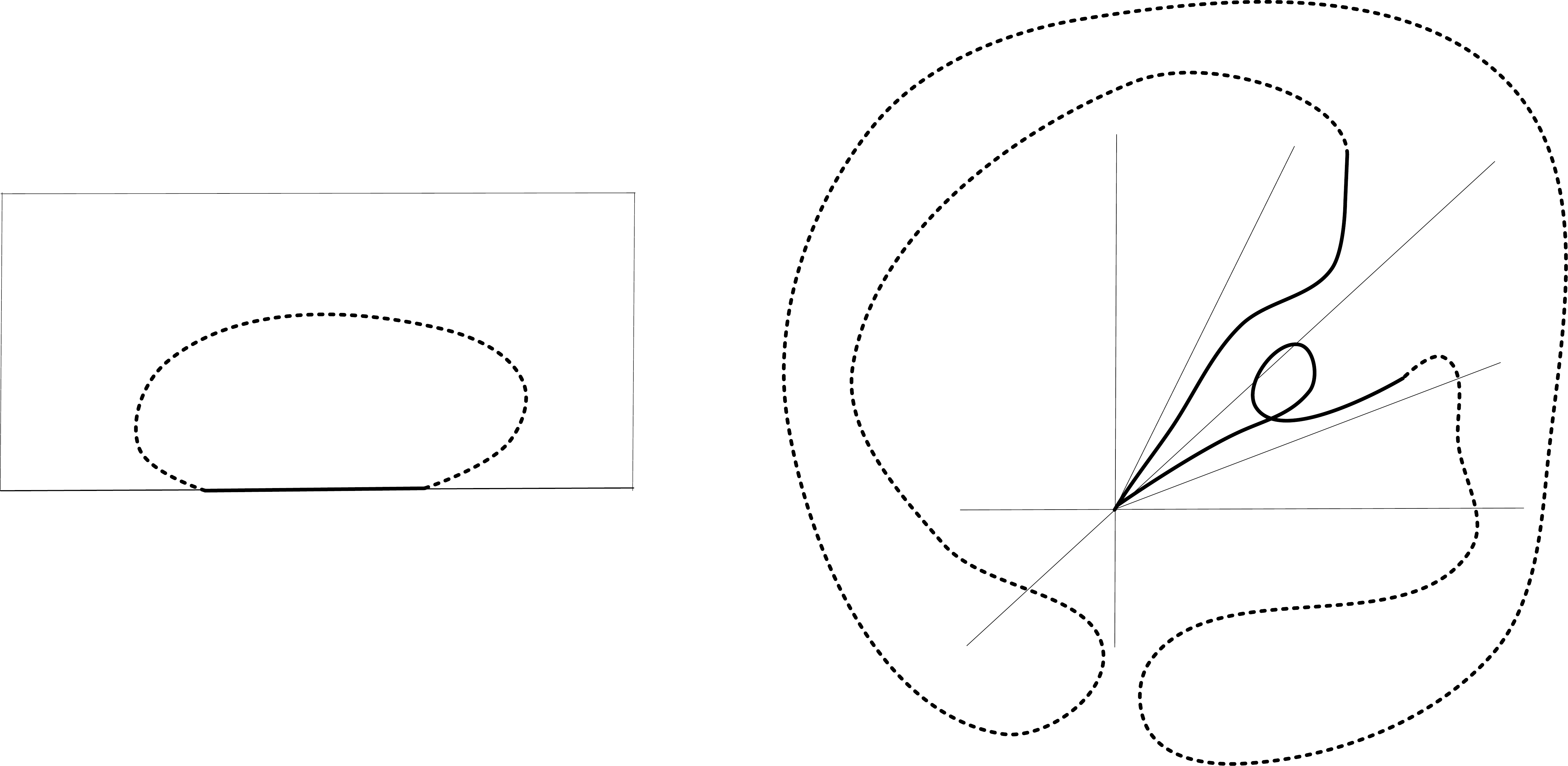}
		\vspace*{0.6cm}
		\caption{On the left is the smooth curve $C \subset (-1,1) \times [0,1)$, and on the right its image in $\bC$ under $\pid \circ u_1$.  The point $(0,0)$ is taken to $0$ under this map, and the thick arc is taken to the region of $\bC$ consisting of points of arguments between $\pi/8$ and $3\pi/8$.}
		\label{fig:smoo1}
	\end{figure}
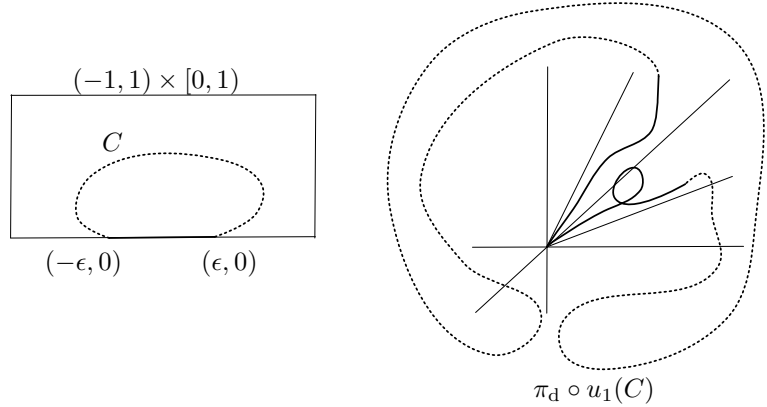
	
	Within $(0,1) \times (-1,1)$, let us add an arc to $\{ 0 \} \times [-\epsilon, \epsilon]$ to make a smooth simple closed loop $C \subset [0,1) \times (-1,1)$.  This is illustrated in Figure \ref{fig:smoo1}.
	
	If we write $Q$ for the first quadrant of the complex plane (consisting of numbers with non-negative real and imaginary parts), and $R$ for the region bounded by $C$, note that we must have $\pid \circ u_1 (R) \cap B \subset Q$ for small enough discs $B \subset \bC$ centered at the origin.  This is because otherwise, for all small enough $\delta > 0$, we would have $\delta e^{-3i\pi/4} \in \pid \circ u_1(R)$.  So for $r$ close enough to $1$, since $\delta e^{-3i\pi/4}$ is in the same component of $\bC^2 \setminus \pid \circ u_r(C)$, we would have $0 \in \pid \circ u_r (R)$.  And hence $u_r ((-1,1) \times (0,1)) \cap \Delta_\bC \not= \emptyset$, contradicting our assumptions.
	
	So let us shrink the region $R$ bounded by $C$ by replacing $C$ with a new smooth closed curve (which we shall also call $C$) made by attaching an arc to $\{ 0 \} \times [-\epsilon, \epsilon]$ such that $\pid \circ u_1 (R) \subset Q$.  We illustrate this in Figure \ref{fig:smoo2}.

	\begin{figure}
		\labellist
		\pinlabel {$C$} at 350 515
		\pinlabel {$(-1,1) \times [0,1)$} at 500 800
		\pinlabel {$(\epsilon,0)$} at 750 175
		\pinlabel {$(-\epsilon,0)$} at 250 175
		\pinlabel {$\pid \circ u_1 (C)$} at 2200 165
		\endlabellist
		\centering
		\includegraphics[scale=0.12]{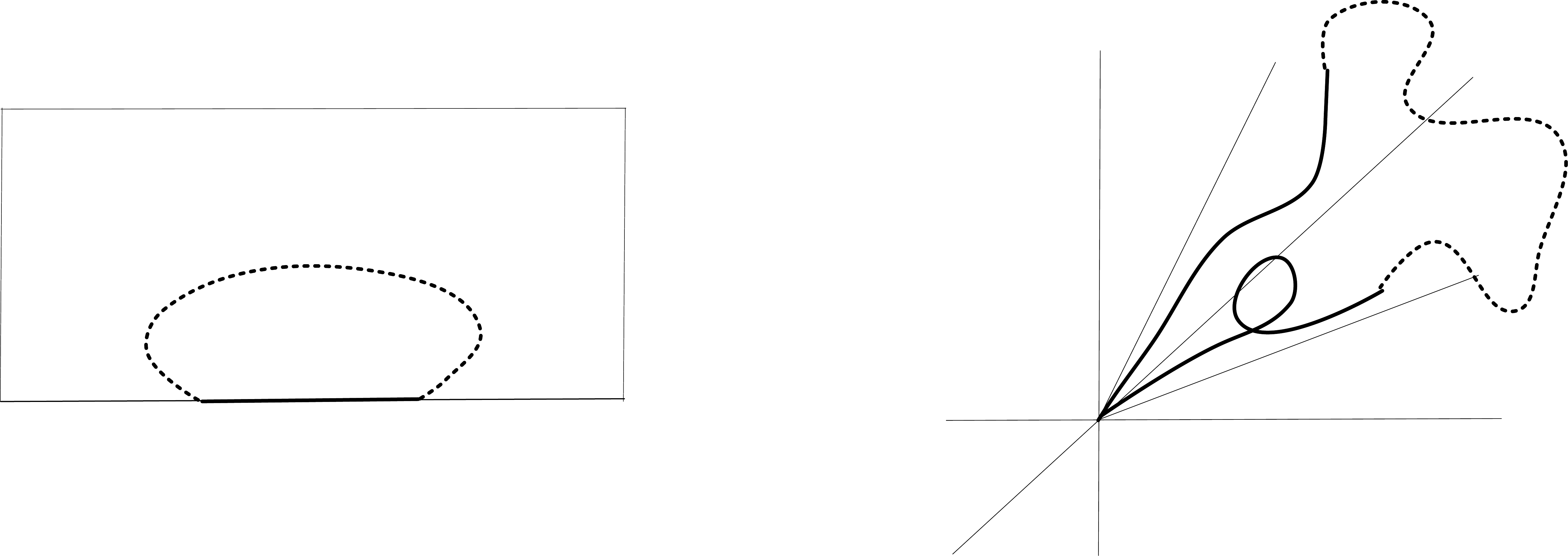}
		\vspace*{0.6cm}
		\caption{By shrinking $C$ we ensure that the region $R$ that it bounds maps into the first quadrant $Q$.}
		\label{fig:smoo2}
	\end{figure}
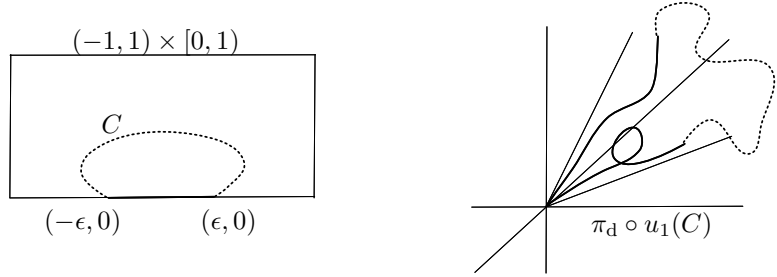

	Since $C$ is smooth, the smooth Riemann mapping theorem gives a diffeomorphism of $R$ with the closed unit disc $D^2$, holomorphic on its interior.  The proposition thus follows from the following lemma.
\end{proof}

\begin{lem}
	\label{lem:hopf}
	There does not exist a smooth map $f \colon D^2 \longrightarrow \bC$, holomorphic on its interior, such that $f(D^2) \subset Q$ (where we write $Q$ for the first quadrant), $f(1) = 0$, and $f$ is non-zero on its interior.
\end{lem}

\noindent The intuition is that the derivative of $f$ should blow up near $1$, since there is a corner point of angle less than $\pi$.  For the following elegant argument we are indebted to Steve Bell.

\begin{proof}
	Suppose that there were such an $f$.  Writing $u$ and $v$ for its real and imaginary parts we see that $uv$ (which is half the imaginary part of $f^2$), is a harmonic function on $D^2$ which attains its minimum value of $0$ at $1$.  The Hopf Lemma then implies that the normal derivative of $uv$ at $1$ is strictly negative.  However, the product rule implies that this derivative vanishes.
\end{proof}

Now we turn to the issue of bubbling at $\Delta_\gamma$.  In terms similar to those in which we phrased the preceding discussion, let us suppose that, for $0 \leq r < 1$
\[ u_r \colon (-1,1) \times [0,1) \longrightarrow \bC^2 \setminus \Delta_\bC \]
is a family of smooth maps
such that
$u_r((-1,1)\times\{0\}) \subset \gamma \times \gamma$, and
such that each $u_r$ is holomorphic on its interior.  We further suppose that the Gromov limit as $r \rightarrow 1$ of these maps is a smooth map
\[ u_1  \colon (-1,1) \times [0,1) \longrightarrow \bC^2 {\rm ,} \]
holomorphic on its interior, meeting the diagonal only once, say at $u_1(0,0) = (p,p) \in \Delta_\gamma$, along with a disc bubbling off
\[ b \colon D^2 \longrightarrow \bC^2 \]
which meets the diagonal again only at the single point $b(1) = (p,p)$.  It is this situation that we now wish to exclude.

\begin{prop}[No diagonal bubbling]
	\label{prop:bubba}
	There is no bubbling at the diagonal.
\end{prop}

\begin{proof}
	Let us analyse the limit $u_1 \# b$ and, similarly to before, consider the projections
	\[ \pid \circ u_1 \colon R \longrightarrow \bC \,\,\, {\rm and} \,\,\, \pid \circ b \colon D^2 \longrightarrow \bC{\rm ,} \]
	where $R \subset (-1,1) \times [0,1)$ is a region biholomorphic to the unit disc, with smooth boundary $C = \partial R \supset [-1/2,1/2] \times \{ 0 \}$.
	Without loss of generality, we assume that $T_p\gamma = e^{i\pi/4}\bR \subset \bC$.

	We choose some small $\epsilon > 0$ so that we have both
	\[ \arg(\pid \circ u_1 ([-\epsilon, \epsilon] \times \{ 0 \})) \subset [\pi/8, 3\pi/8] \cup [-7\pi/8, -5\pi/8] {\rm ,} \]
	and
	\[ \arg(\pid \circ b  (\{ e^{i\theta} : -\epsilon \leq \theta \leq \epsilon \}) ) \subset [\pi/8, 3\pi/8] \cup [-7\pi/8, -5\pi/8] {\rm .} \]

	Focussing on $u_1$, consider the case that $\arg(\pid \circ u_1 ([-\epsilon, \epsilon] \times \{ 0 \}))$ is contained in just one of the intervals on the right-hand side of this statement, say without loss of generality
	\[ \arg(\pid \circ u_1 ([-\epsilon, \epsilon] \times \{ 0 \})) \subset [\pi/8, 3\pi/8] {\rm .}\]
		
	(Then we also necessarily have
	\[ \arg(\pid \circ b  (\{ e^{i\theta} : -\epsilon \leq \theta \leq \epsilon \}) ) \subset [\pi/8, 3\pi/8] {\rm ,} \]
	since the maps $u_r$ which limit to $u_1 \# b$ do not meet the diagonal $\Delta_\bC$.)
	
	Now our argument before (using the Hopf Lemma) applies again to rule out any restriction of $\pid \circ u_1$ to a disc with smooth boundary including $(0,0)$ that maps to the first quadrant $Q \subset \bC$.  Again, this means that for all small enough $\delta > 0$, we have $\delta e^{-3i\pi/4} \in \pid \circ u_1(R)$; and this contradicts the assumption that $u_r$ does not meet $\Delta_{\bC}$ as before.
	
	On the other hand, in the complementary case, we must have, for $\delta > 0 $ small enough, either that $\delta e^{-i\pi/4} \in \pid \circ u_1 ((-1,1) \times (0,1))$ or that $\delta e^{3i\pi/4} \in \pid \circ u_1 ((-1,1) \times (0,1))$.  Then, since for $r$ close to $1$ we have that $\delta e^{-i \pi/4}$ and $\delta e^{\i \pi/4}$ lie in the same component of $\bC \setminus \pid \circ u_r (C)$ as does $0$, we must have that the image of $u_r$ meets the diagonal $\Delta_\bC$, contradicting our assumptions.
\end{proof}
\subsection{Duality.}
\label{subsec:dualityofJF}
In this subsection we establish a duality result for Jordan Floer homology.  That such a result should exist is unsurprising: at least generically, the Jordan Floer complex $\JFC(\gamma, \theta)$ for $0 < \theta < \pi$ is generated by inscriptions of a $\theta$-rectangles in $\gamma$.  But of course, inscribed $\theta$-rectangles are exactly $(\pi - \theta)$-rectangles.  So we would expect to see some kind of relationship between the two complexes, and thus a correspondence between $\JF(\gamma, \theta)$ and $\JF(\gamma, \pi-\theta)$, and their spectral invariants.

\begin{prop}
	\label{prop:duality}
	If $\gamma$ is a smooth Jordan curve and $0 < \theta < \pi$, then
	\[ \ell_1(\gamma, \theta) = \area(\gamma) - \ell_2(\gamma, \pi - \theta) {\rm .} \]
\end{prop}

\begin{proof}
	The Jordan Floer chain complex $\JFC(\gamma, \theta)$ is generated by trajectories of a certain Hamiltonian vector field that begin and end on $\gamma \times \gamma$.  An alternative point of view is that the complex is generated by the transverse intersection points of $\gamma \times \gamma$  with its preimage under the time $1$ Hamiltonian flow.\footnote{In considering the preimage here we follow Auroux's notes \cite{aurouxintro}.}  In this latter point of view, the differential counts strips limiting at their ends to intersection points, and with boundary conditions on both Lagrangians.
	
	Generically, the time $1$ Hamiltonian flow is given by the rotation
	\[ R_{-\theta} \colon \bC^2 \longrightarrow \bC^2 \colon (z,w) \longmapsto \left( \frac{z+w}{2} + e^{-i \theta} \frac{z-w}{2}, \frac{z+w}{2}  - e^{-i \theta} \frac{z-w}{2} \right) {\rm ,} \]
	which is a rotation by angle $-\theta$ that leaves the diagonal $\Delta_\bC$ fixed.  So the alternative point of view is to consider intersection points between $\gamma \times \gamma$ and $R_\theta(\gamma \times \gamma)$.
	
	The genericity condition is that the Lagrangians $\gamma \times \gamma$ and $R_\theta(\gamma \times \gamma)$ intersect transversally away from $\Delta_\bC$.  Assuming for a moment that this is satisfied, we notice that the symplectomorphism $R_{\pi - \theta}$ takes the Lagrangian $\gamma \times \gamma$ to $R_{\pi - \theta}(\gamma \times \gamma)$ and takes the Lagrangian $R_\theta(\gamma \times \gamma)$ to $R_\pi (\gamma \times \gamma) = \gamma \times \gamma$.  Thus we see that in this case, assuming that we pick almost complex structures compatibly under $(R_{\pi - \theta})_*$, the complexes $\JFC(\gamma, \theta)$ and $\JFC(\gamma, \pi - \theta)$ are genuinely dual -- there is a correspondence between intersection points
	\[ p \in (\gamma \times \gamma) \cap R_\theta(\gamma \times \gamma) \longleftrightarrow p' = R_{\pi - \theta}(p) \in (\gamma \times \gamma) \cap R_{\pi - \theta}(\gamma \times \gamma) \]
	and furthermore if $u \colon \Sigma \longrightarrow \bC^2$ is a strip connecting $p$ to $q$ for $p,q \in (\gamma \times \gamma) \cap R_\theta(\gamma \times \gamma)$, then $u' = R_{\pi - \theta} \circ u$ is a strip connecting $q'$ to $p'$.
	
	It is immediate that the relative differences in action and in Maslov index between $p$ and $q$ agree with the relative differences between $q'$ and $p'$.  In fact, the absolute Maslov index $\mu(p')$ is $3 - \mu(p)$, since $\JF$ is of dimension $1$ in degrees $1$ and $2$, and trivial outside these degreees.  Moreover, we have that the actions satisfy $\cA(p) + \cA(p') = \area(\gamma)$, as we now show.
	
	The action of a path is expressed as the sum of two terms: one is the integral of the Hamiltonian along the trajectory, while the other is the symplectic area of a `capping' disc cobounded by the Lagrangian (in this case $\gamma \times \gamma$) and the path.  In the definition of the Jordan Floer chain complex, these latter discs were taken to have trivial intersection with $\Delta_\bC$.
	
	Suppose that $\tau_p$ and $\tau_{p'}$ are the trajectories corresponding to $p$ and its dual $p'$.  They can each be considered to be flowlines associated to the Hamiltonian
	\[ H \colon \bC^2 \longrightarrow \bR \colon (z,w) \longmapsto \frac{\vert z - w \vert^2}{4} {\rm ,}\]
	with $\tau_p \colon [0,\theta] \longrightarrow \bC^2$ being a flowline for time $\theta$ and $\tau_{p'} \colon [0, \pi - \theta] \longrightarrow \bC^2$ being a flowline for time $\pi - \theta$.
	Furthermore $R_{-\theta} \tau_{p'}(0) = \tau_p(\theta)$ and $R_{-\theta}(\tau_{p'}(\pi - \theta)) \in \gamma \times \gamma$.  So if we glue together the flowlines $\tau_p$ and $R_{-\theta}(\tau_{p'})$, we obtain a time $\pi$ flowline $\tau$ that starts and ends on $\gamma \times \gamma$ (as it has to since $R_{- \pi} = R_\pi$ is an involution on $\bC^2$ satisfying $R_{\pi} (\gamma \times \gamma) = \gamma \times \gamma$).
	
	We can also glue together the two cobounding discs to obtain a singular disc $D \subset \bC^2$ which is cobounded by $\tau$ and $\gamma \times \gamma$, and that is disjoint from $\Delta_\bC$.  We have
	\[ \cA(p) + \cA(p') = \int_0^\pi H \circ \tau(t) dt - \int_D \omega {\rm ,} \]
	so that
	\[ 2(\cA(p) + \cA(p')) = \int_0^{2 \pi} H \circ (\tau \sqcup R_{\pi} \tau)(t) dt - \int_{D \sqcup R_{\pi}(D)} \omega \]
	where we write $\tau \sqcup R_\pi \tau$ for the concatenation of $\tau$ and $R_\pi \circ \tau$.
	Now the first integral is just the symplectic area of the planar disc bounded by the circle $\tau \sqcup R_\pi \tau$, so that the whole expression is the symplectic area bounded by the curve formed by the parts of the boundary of $D \sqcup R_\pi(D)$ lying on $\gamma \times \gamma$.  Now curves lying on $\gamma \times \gamma$ which avoid the diagonal $\Delta_\bC$ fall into isotopy classes depending on their winding number around $\bC$.  In this case, the winding number around $\bC$ is $1$, since the planar disc bounded by $\tau \sqcup R_\pi \tau$ meets $\bC$ once transversely.  The symplectic area bounded by each curve in a given isotopy class is proportional to this winding number.  Thus we see that the whole expression evaluates to $2 \area(\gamma)$, and we are done.

	If one does not assume genericity of the Jordan curve $\gamma$, then the Jordan Floer homology $\JF(\gamma,\theta)$ is defined to be the directed limit of the homologies of Jordan Floer complexes that include Hamiltonian perturbation terms, the limit being taken as the perturbations tend to $0$.  We shall show that small Hamiltonians $h$ applied to the `$\theta$' Lagrangians correspond to small Hamiltonians $-\wt{h}$ which can be applied to the `$\pi - \theta$' Lagrangians, resulting in a dual complex.
	
	Suppose, then, that one achieves transversality of $\gamma \times \gamma$ and $R_\theta(\gamma \times \gamma)$ away from $\Delta_\gamma$ by applying the Hamiltonian perturbation $\phi_h \colon \bC^2 \longrightarrow \bC^2$ which is the time $1$ flow of a small Hamiltonian $h \colon \bC^2 \longrightarrow \bR$ supported away from $\Delta_\bC$.  (For convenience, in this discussion we are assuming that we achieve transversality by applying $\phi_h$ following $R_\theta$, and that $h$ is time-independent).
	
	In other words, $\gamma \times \gamma$ and $\phi_h \circ R_\theta (\gamma \times \gamma)$ are transverse away from $\Delta_\gamma$.  Then we see that $R_{\pi - \theta} (\gamma \times \gamma)$ and $R_{\pi - \theta} \circ \phi_h \circ R_\theta (\gamma \times \gamma)$ are also transverse away from $\Delta_\gamma$.
	
	Now $R_{\pi - \theta} \circ \phi_h \circ R_\theta = \phi_{\wt{h}} \circ  R_\pi$ where we define the Hamiltonian $\wt{h} = h \circ R_{\theta - \pi} \colon \bC^2 \longrightarrow \bR$.  Since $\gamma \times \gamma$ is invariant under $R_\pi$, we have that $R_{\pi - \theta} (\gamma \times \gamma)$ and $\phi_{\wt{h}}(\gamma \times \gamma)$ are transverse away from $\Delta_\gamma$.
	Equivalently, $\gamma \times \gamma$ and $\phi_{-\wt{h}} \circ R_{\pi - \theta} (\gamma \times \gamma)$ are transverse away from $\Delta_\gamma$.
	
	So $\phi_{-\wt{h}}$ is a small Hamiltonian perturbation that can be applied to achieve transversality of $\gamma \times \gamma$ and $R_{\pi - \theta} (\gamma \times \gamma)$.  Applying this perturbation results in a pair of Lagrangians, transverse away from $\Delta_\gamma$, that is symplectomorphic (via $\phi_{\wt{h}}$) to the pair  $\phi_{\wt{h}}(\gamma \times \gamma)$, $R_{\pi - \theta} (\gamma \times \gamma)$.  But this pair becomes the pair $\phi_h \circ R_\theta (\gamma \times \gamma)$, $\gamma \times \gamma$ after applying the rotation $R_{\theta - \pi}$ since
	\[ R_{\theta - \pi} \circ \phi_{\wt{h}}(\gamma \times \gamma) = \phi_{\wt{h} \circ R_{\pi-\theta}} \circ R_{\theta-\pi}(\gamma \times \gamma) = \phi_{h} \circ R_{\theta}(\gamma \times \gamma) {\rm .} \]
	
	Finally we have arrived at a familiar situation, namely the Lagrangian pair $\gamma \times \gamma$, $R_\theta(\gamma \times \gamma)$ (now perturbed by $\phi_h$) is symplectomorphic to the Lagrangian pair $R_{\pi - \theta}(\gamma \times \gamma)$, $\gamma \times \gamma$ (now perturbed by $\phi_{-\wt{h}}$).  %In fact, it is readily verified that
	These perturbed complexes are again precisely dual to each other (given compatible choices of almost complex structure) with dual actions and Maslov indices satisfying the same identities.
\end{proof}

\section{Deformations of the Jordan curve.}
\label{sec:deformation}
In this section we establish firstly that the spectral invariants are continuous with respect to deformations of the Jordan curve in Subsections \ref{subsec:cty} and \ref{subsec:bifurcationmethod}.  In Subsection \ref{subsec:hoferdistancebounds} we refine the statement of continuity by bounding the variation of the spectral invariant in terms of the Hofer norm.  In Subsection \ref{subsec:pwlinear} we turn to piecewise linear Jordan curves, giving a definition of the spectral invariants associated to such curves and verifying continuity of the spectral invariant in isotopies through these curves.  Finally in Subsection \ref{subsec:grace} we turn to isotopies of nested Jordan curves in which each curve only admits elegant inscriptions of $\theta$-rectangles, and bound the variation of the spectral invariants in this case.
\subsection{Hamiltonian isotopy of the Lagrangians.}
\label{subsec:cty}

Suppose that $\gamma_r \subset \bC$ for $0 \leq r \leq 1$ is a smooth $1$-parameter family of smooth Jordan curves, each of constant area $A = \area(\gamma_r)$.
\begin{lem}
	\label{lem:equiareal_hamiltonian}
	There exists a $1$-parameter family of smooth Hamiltonians
	\[ H_r \colon \bC \longrightarrow \bR \]
	with associated Hamiltonian flows
	\[ \phi_r \colon \bC \longrightarrow \bC \]
	such that $\phi_0$ is the identity and $\phi_r(\gamma_0) = \gamma_r$ for all $r \in [0,1]$.
\end{lem}

\Cref{lem:equiareal_hamiltonian} and its proof are inspired by the result of Akveld-Salamon that a pair of equal area disks on a symplectic 2-manifold are related by a Hamiltonian isotopy \cite[Proposition A.1]{akveldsalamonart}.
The argument follows standard lines (Moser's method), and we give details for the non-expert.

\begin{proof}
By the isotopy extension principle, there exists a smooth 1-parameter family $f_r \in \mathrm{Diff}(\bC)$, $r \in [0,1]$, such that $\gamma_r = f_r(\gamma_0)$.
Let $\omega$ denote the standard symplectic form on $\bC$.
Define $\omega_r := f_r^* \omega$ for all $r$.
Let $D$ denote the domain bounded by $\gamma_0$.

We will produce a smooth 1-parameter family $\psi_r \in \mathrm{Diff}(\bC)$ for $0 \leq r \leq 1$, such that
\begin{equation}
\label{e: psi_t}
\psi_r(D) =D \quad \textup{and} \quad \psi_r^* \omega_r = \omega \quad  {\rm for} \quad 0 \leq r \leq 1.
\end{equation}
Then $\phi_r := f_r \circ \psi_r$ satisfies
\[
\phi_r(D) = f_r \circ \psi_r(D) = f_r(D) \quad \textup{and} \quad \phi_r^* \omega = \psi_r^* f_r^* \omega = \psi_r^* \omega_r = \omega \quad {\rm for} \quad 0 \leq r \leq 1.
\]
So $\phi_r$ is a symplectic isotopy of $(\bR^2,\omega)$ such that $\phi_r(\gamma_0) = \gamma_r$ for $0 \leq r \leq 1$.

Let $X_r$ denote the vector field which generates the flow $\phi_r$.
The 1-form $\iota(X_r)\omega$ is exact, since $H^1(\bC) = 0$, so there exists a unique 1-parameter family of smooth functions $H_r : \bC \to \bR$ satisfying $H_r(0) = 0$ and $d H_r = \iota(X_r) \omega$ for $0 \leq r \leq 1$.
It follows that $\phi_r$ is a Hamiltonian isotopy with associated Hamiltonian $H_r$.

It remains to construct $\psi_r$.  Since $H^2(\bC) = 0$, there exists a 1-form $\alpha_r$ such that $d \alpha_r = \omega - \omega_r$ for $0 \leq r \leq 1$.
If we place a Riemannian metric on $\bC$, then we may specify $\alpha_r$ uniquely by the additional condition that it is harmonic.
Chosen in this manner, the forms $\alpha_r$ vary smoothly with $r$.

The assumption that each $\gamma_r$ encloses the same area translates into the assertion that
\[
\int_D \omega_r = \int_D \omega, \quad {\rm for} \quad 0 \leq r \leq 1.
\]
By Stokes's theorem, we have
\[
\int_{\del D} \alpha_r = \int_D d \alpha_r = \int_D (\omega - \omega_r) = 0, \quad {\rm for} \quad 0 \leq r \leq 1.
\]
It follows that $\alpha_r | (\del D)$ is exact for all $r$.
Hence $\alpha_r | (\del D) = d g_r$ for a smooth function $g_r : \del D \to \bR$ which is uniquely determined by the condition that $g_r(p) = 0$ for a choice of basepoint $p \in \del D$.
Moreover, the family of functions $g_r$ vary smoothly in $r$.
We extend $g_r$ to a smooth family of smooth functions $g_r : \bC \longrightarrow \bR$.
We then set $\beta_r = \alpha_r - d g_r$ for all $0 \leq r \leq 1$.
Then
\[
d \beta_r = \omega - \omega_r \quad \textup{and} \quad \beta_r |(\del D) \equiv 0, \quad {\rm for} \quad 0 \leq r \leq 1.
\]

Define
\[
\omega_{rt} = t \cdot \omega_r + (1-t) \cdot \omega \quad {\rm for} \quad  0 \leq r,t \leq 1.
\]
Observe that each $\omega_{rt}$ is a convex combination of positive area forms, so it is again a positive area form, and hence a symplectic form.
By nondegeneracy, it follows that we can uniquely define a vector field $X_{rt} \in \mathrm{Vect}(\bC)$, smooth in $0 \leq r,t \leq 1$, by the condition
\[
\iota(X_{rt}) \omega_{rt} = \beta_r.
\]
We define in turn a flow $\psi_{rt} \in \mathrm{Diff}(\bC)$ by the differential equation
\[
\frac{d}{dt} \psi_{rt} = X_{rt} \cdot \psi_{rt}
\]
with initial condition $\psi_{r0} = \mathrm{id}$ for $0 \leq r \leq 1$.

We shall show that $\psi_r := \psi_{r1}$ defines the desired isotopy.

First, observe that for a non-zero tangent vector $v$ to $\del D$, we have
\[
0 = \beta_t(v) = \iota(X_{rt})\omega_{rt}(v) = \omega_{rt}(X_{rt},v).
\]
Since $\omega_{rt}$ is nondegenerate and $\bC$ is 2-dimensional, it follows that $X_{rt}$ is proportional to $v$.
That is, $X_{rt}$ is tangent to $\del D$ for all $r$ and $t$; so it follows that the flow $\psi_{rt}$ preserves $\del D$.
Hence $\psi_{rt}(D) = D$ for all $r$ and $t$.
Taking $t=1$, we obtain $\psi_r(D) = D$ for all $t$, which gives the first part of \eqref{e: psi_t}.

Second, we shall show that
\begin{equation}
\label{e: derivative vanishes}
\frac{d}{dt} \psi_{rt}^* \omega_{rt} = 0.
\end{equation}
Assuming this is the case, also note that we have $\psi_{r0}^* \omega_{r0} = \mathrm{id}^* \omega = \omega$ for all $r$.
Since the derivative in $t$ vanishes, it follows that $\psi_r^* \omega_r = \psi_{r1}^* \omega_{r1} = \omega$ as well for all $r$, which gives the second part of \eqref{e: psi_t}. 

Hence we are left to verify equation \eqref{e: derivative vanishes}.
We have
\[
\frac{d}{dr} \psi_{rt}^* \omega_{rt} = \psi_{rt}^* \left( \mathcal{L}_{X_{rt}} \omega_{rt} + \frac{d \omega_{rt}}{dr} \right).
\]
Cartan's formula gives
\[
\mathcal{L}_{X_{rt}} \omega_{rt} = d \iota(X_{rt})\omega_{rt} + \iota(X_{rt}) d \omega_{rt}.
\]
The term $d \iota(X_{rt}) \omega_{rt}$ becomes $d \beta_r = \omega - \omega_r$ on substitution, while the second term vanishes, since $\omega_{rt}$ is closed.
On the other hand,
\[
\frac{d \omega_{rt}}{dr} = \frac{d}{dt} (t \cdot \omega_r + (1-t) \cdot \omega) = \omega_r - \omega.
\]
Consequently, the terms cancel, and we obtain \eqref{e: derivative vanishes}.
\end{proof}

It follows that (abusing notation a little)
\[ H_r \colon \bC^2 \longrightarrow \bR \colon (z,w) \longmapsto H_r(z) + H_r(w) \]
is a Hamiltonian whose associated flow $\Phi_r$ gives an isotopy through the family of Lagrangians $\gamma_r \times \gamma_r = \Phi_r(\gamma_0 \times \gamma_0)$.

Now the pair of Lagrangians
\[ (\gamma_r \times \gamma_r, R_\theta(\gamma_r \times \gamma_r)) = (\gamma_r \times \gamma_r, R_\theta \circ \Phi_r (\gamma_0 \times \gamma_0)) \]
is Hamiltonian isotopic (in particular, symplectomorphic) to
\[ (\gamma_0 \times \gamma_0, \Phi_r^{-1} \circ R_\theta \circ \Phi_r (\gamma_0 \times \gamma_0)) {\rm .} \]
We wish to consider the Jordan Floer homology of this pair of Lagrangians.  In particular we wish to study how the Jordan Floer homology changes when $\gamma_0 \times \gamma_0$ is the first Lagrangian, and we vary the second Lagrangian in this pair through the $1$-parameter family of Hamiltonian isotopic Lagrangians $\Phi_r^{-1} \circ R_\theta \circ \Phi_r (\gamma_0 \times \gamma_0)$.

It will be convenient, and more intuitive, for us to apply the symplectomorphism $\Phi_r$ above, so as rather to consider varying both Lagrangians at the same time so that we move instead through the family of pairs
\[ (\gamma_r \times \gamma_r, R_\theta(\gamma_r \times \gamma_r)) {\rm ,} \]
in which both the second and the first Lagrangian vary with $r$.

\subsection{The bifurcation method.}
\label{subsec:bifurcationmethod}

There are two principal approaches to studying Floer homology under Hamiltonian isotopy.
The more recent approach is that of \emph{continuation}.
It was this approach that we followed in \cite{greenelobb3} when considering well-definedness of Jordan Floer homology and variation in the angle $\theta$.  In the current paper we shall be following the \emph{bifurcation} approach, the method originally used by Floer \cite{floerlag} when defining Lagrangian Floer homology.  We take this approach firstly for technical reasons and secondly since our principal interest is at the chain level variation of the action, and the bifurcation method brings this to the fore.

We shall begin by giving an exposition of the bifurcation method, adapted to our setting in which strips are required to avoid the diagonal $\Delta_\bC$, and follow up with Proposition \ref{prop:cts_in_deformation}, whose proof largely follows from the exposition.

\subsubsection*{Exposition of the method.}
The bifurcation method is based on analogy with Cerf theory.  If one has two Morse-Smale vector fields on a closed manifold, one cannot always connect them by a path of Morse-Smale vector fields.  But one can find a path which is Morse-Smale except at finitely many times; these finitely many times corresponding to handleslides (when the Morse-Smale condition fails in the weakest way possible) or the birth-death of a pair of critical points.

Following Floer, the bifurcation method proceeds as follows, firstly under the assumption that, for all $r$, $\gamma_r \times \gamma_r$ is transverse $R_\theta(\gamma_r \times \gamma_r)$ away from $\Delta_\gamma$.  This assumption implies that each transverse intersection point
\[ p_0 \in (\gamma_0 \times \gamma_0) \cap R_\theta(\gamma_0 \times \gamma_0) \]
extends to a path of transverse intersections
\[ p_r \in (\gamma_r \times \gamma_r) \cap R_\theta(\gamma_r \times \gamma_r) {\rm .} \]

Given a choice of $1$-parameter family of $t$-dependent almost complex structures $J^r_t$, we shall consider pairs $(r,u)$ where $0 \leq r \leq 1$, and $u^r \colon \Sigma \longrightarrow \bC$ is a bounded energy strip satisfying a Lagrangian boundary condition and a Cauchy-Riemann equation which depend on $r$.

Specifically, we shall consider smooth maps
\[
u^r \colon \Sigma \longrightarrow \bC^2
\]
which satisfy the boundary conditions
\begin{equation}
	\tag{BC}
	\begin{cases}
		u^r(\bR \times \{0\}) \subset \gamma_r \times \gamma_r, \\[.2cm] u^r(\bR \times \{1\}) \subset R_\theta(\gamma_r \times \gamma_r);
	\end{cases}
\end{equation}
which satisfy the Cauchy-Riemann-Floer equation
\begin{equation}
	\label{e:CRF}
	\tag{CRF}
	\del_s u^r + J^r_t \del_t u^r= 0;
\end{equation}
and which have bounded energy
\begin{equation}
	\tag{BE}
	\int |\del_s u^r|_{J^r_t}^2 d s \, d t < \infty.
\end{equation}

It follows from (BE) that $u^r$ extends uniquely to a continuous function on the extended strip $\overline{\Sigma} = [-\infty,\infty] \times [0,1]$, where $[-\infty,\infty]$ is topologized as the end compactification of $\bR = (-\infty,\infty)$.

\begin{defin}
	[Moduli spaces of strips]
	\label{defin:moduli_space_strips}
	Suppose that $p,q$ are paths of transverse intersection points as above.
	We define the moduli space of strips
	\begin{align*}
		\cM^\Delta(p,q) = \{ (r,u^r) : 0 \leq r \leq 1, \,\, & u^r \in C^\infty(\Sigma, \bC^2) : u^r \textup{ obeys (BC), (CRF), (BE)}, \\ & u^r({\overline{\Sigma}}) \cap \Delta(\bC) = \emptyset, \\ & \lim_{s \rightarrow -\infty} u^r(s,t) = p_r, \lim_{s \rightarrow +\infty} u^r(s,t) = q_r \},
	\end{align*}
	and we write $\ol{\cM}^\Delta(p,q)$ for the quotient of this moduli space by the $\bR$-action corresponding to translation along the $s$ coordinate of $\Sigma$.
\end{defin}

Floer \cite{floerlag} showed that this parametrized version of the holomorphic strips used to define Lagrangian Floer homology allows admissable choices of parametrised almost complex structures $J^r_t$.  For such admissable choices, $\ol{\cM}^\Delta(p,q)$ is a manifold of dimension equal to the difference in Maslov indices $\mu(p) - \mu(q)$ (where we understand the Maslov index of any representative transverse intersection point along the path).

For our present purposes, we shall restrict to admissable parametrized almost complex structures $J^r_t$ that agree with the standard almost complex structure on the diagonal $\Delta_\bC$, and that agree with standard almost complex structure when near the boundaries of $\Sigma$ and within a small neighborhood of $\Delta_\bC$.\footnote{We make precise choices of these neighborhoods in \cite{greenelobb3}.  Not wishing to clutter the current exposition by repeating definitions, we allow ourselves now a degree of imprecision; exact details may be transferred \emph{mutatis mutandis} from our earlier work.}  Such perturbations suffice for transversality as explained, for example, by Audin-Damian \cite{audindamian}.  Let us assume that this restriction has been made from now on.

Proofs given in \cite{greenelobb3} show that a $1$-parameter family of strips in a space $\cM^\Delta(p,q)$ cannot degenerate by sphere bubbling, or limit to a strip meeting the diagonal in its interior, or limit to a so-called \emph{broken} strip which breaks at a point of $\Delta_\bC$.  To be more precise, these arguments were not given in the context of a family of pairs $(r,u^r)$, but they readily apply with only cosmetic modification.  In Proposition~\ref{prop:boundary_touching} we have supplied an argument (again admitting cosmetic modification due to the dependence on $r$) which rules out a limit strip meeting $\Delta_\gamma$ on its exterior, and in Proposition \ref{prop:bubba} we have ruled out disc bubbling.  Hence we find ourselves able to repeat Floer's argument from \cite{floerlag}, which we now summarize, specializing to our situation.

If $\mu(p) - \mu(q) = 1$ and $(r,u_r) \in \cM^\Delta(p,q)$, then $u_r$ is a strip of Maslov index $1$ that connects $p_r$ to $q_r$, and so contributes to the coefficient of $q_r$ in the Floer differential applied to $p_r$.  If the only boundaries of all components of the compactifications of $\cM^\Delta(p,q)$ occur at $r=0$ and $r=1$, then we see that the Floer differential applied to $p_0$ has the same coefficient of $q_0$ as the coefficient of $q_1$ in the Floer differential applied to $q_1$.

However, there may be broken strips of the form
\[  (r, v_r) \star (r, u_r) \in \cM(p, p_0) \times \cM(p_0, q)  \,\, {\rm or} \,\, (r, u_r) \star (r, w_r) \in \cM^\Delta(p, q_0) \times \cM^\Delta(q_0, q) \]
where $\mu(p_0) = \mu(p)$ and $\mu(q_0) = \mu(q)$.  Floer showed that when, for example, $(r, u_r) \in \cM^\Delta(p_0,q)$, this gives rise to a chain complex isomorphism at time $r$, essentially replacing the basis element $p$ with $p + p_0$.  The analogy with Cerf theory is a $1$-parameter family of Morse-Smale vector fields passing through a non-Morse-Smale vector field at time $r$ in which there is a single flow-line connecting two critical points of the same index.  This corresponds, in the handle decomposition picture, to a handle-slide.

So far, we have considered the case in which, for all $0 \leq r \leq 1$, $\gamma_r \times \gamma_r$ has only transverse intersections with $R_\theta(\gamma_r \times \gamma_r)$ away from their clean intersection $\Delta_\gamma$.  In the general case, Floer showed that one can, by a small $r$-dependent Hamiltonian perturbation, achieve transversality off a finite number of `degenerate' $r$ values, while at and around the degenerate $r$, the allowable intersection loci admit an explicit local description.

To provide a little more detail, the degenerate intersection $p_{r_0} = q_{r_0}$ occurring at a degenerate $r_0$ corresponds to the cancellation (or, for the more optimistic, the creation) of a pair of transverse intersection points ($p_r$ and $q_r$ for $r < r_0$, say) of Maslov indices satisfying $\mu(p_r) - \mu(q_r) = 1$.  Floer's analysis shows that for $r < r_0$ and close to $r_0$, there is a unique holomorphic strip from $p_r$ to $q_r$.
Floer's argument works under the assumption that $\pi_2(M) = \pi_1(L) = 0$, so there is an absolute action functional on generators.
This forces any strip from $p_r$ to $q_r$ to have small energy $\cA(p_r)-\cA(q_r)$.
Floer argues that any such strip has image contained in a suitably small neighborhood of the bifurcation point, and in turn is uniquely determined.
The argument applies in our particular set-up as well, since we also have an absolute action functional.
For $r > r_0$ and close to $r_0$, the Floer complex is the result of the Gaussian elimination of the generators $p$ and $q$ in the Floer complex for $r < r_0$ and close to $r_0$

\subsubsection*{Continuity of spectral invariants in deformation of the Jordan curve.}  We are now ready to draw the conclusion that we desired.

\begin{prop}
	\label{prop:cts_in_deformation}
	Suppose that $\gamma_r \subset \bC$ is a smooth family of Jordan curves for $0 \leq r \leq 1$, $i \in \{ 1 , 2 \}$, and $0 < \theta < \pi$.  Then the function
	\[ \ell_i(\gamma_{\cdot}, \theta) \colon [0,1] \longrightarrow \bR \colon r \longmapsto \ell_i(\gamma_r,\theta)  \]
	is continuous.
\end{prop}

\begin{proof}
	We define the smooth family of Jordan curves
	\[ \gamma'_r := \frac{1}{\area(\gamma_r)} \gamma_r \]
	which each bound regions of area $1$.  Then we know by Lemma \ref{lem:equiareal_hamiltonian} that this family is Hamiltonian isotopic, which means, following the discussion above, that we can apply Floer's bifurcation argument to the Jordan Floer chain complexes $\JFC(\gamma'_r, \theta)$.
	
	There are finitely many critical values $0 < r_1 < \cdots < r_k < 1$ corresponding to handleslides or handle cancellations.
	At each one, the bifurcation argument establishes a chain homotopy equivalence which is filtered of degree 0 and so preserves the spectral invariants.
	For $r$ varying between a pair of consecutive critical values, it instead gives a chain complex isomorphism.
	This isomorphism carries standard generators to standard generators, and their actions vary continuously in $r$.
	Hence the spectral invariants vary continuously between a pair of consecutive critical values.
			
	So we see that
	\[ [0,1] \longrightarrow \bR : r \longmapsto \ell_i(\gamma'_r, \theta) \]
	is continuous.  But the complex $\JFC(\gamma_r, \theta)$ is isomorphic to $\JFC(\gamma'_r, \theta)$ with rescaled action, so we can conclude that
	\[ \ell_i(\gamma_{\cdot}, \theta) \colon [0,1] \longrightarrow \bR \colon r \longmapsto \ell_i(\gamma_r,\theta) = \area(\gamma'_r,\theta) \ell_i(\gamma'_r,\theta) \]
	is also continuous.
\end{proof}
\subsection{Hofer distance bounds.}
\label{subsec:hoferdistancebounds}
In this subsection we refine the continuity statement of Proposition~\ref{prop:cts_in_deformation} by considering the Hofer distance between Hamiltonian isotopic Jordan curves.

Suppose that $\gamma_0, \gamma_1 \subset \bC$ are smooth Jordan curves that bound the same area.  We have seen in Lemma \ref{lem:equiareal_hamiltonian} that they are Hamiltonian isotopic.  For $0 \leq t \leq 1$, let
\[ h_t \colon \bC \longrightarrow \bR \]
be a smooth bounded function, identically $0$ for all $t$ close to $0$ or $1$, such that the time 1 flow $\phi^1 \colon \bC \longrightarrow \bC$ of the $t$-dependent Hamiltonian vector field $X_t$ associated to $h_t$ satisfies $\phi^1(\gamma_0) = \gamma_1$.  
\begin{defin}
	\label{defin:hofernorm}
	The \emph{Hofer norm} of $h_t$ is
	\[ \vert h_t \vert := \int_0^1 ( \sup(h_t) - \inf(h_t) ) d t {\rm ,} \]
	while the \emph{Hofer distance} $\d (\gamma_0, \gamma_1)$ between $\gamma_0$ and $\gamma_1$ is the infimum of the set of all Hofer norms of such Hamiltonians.
\end{defin}

We are going to consider the path of smooth Jordan curves $\gamma_r$ for $0 \leq r \leq 1$, where $\gamma_r$ is the image of $\gamma_0$ under the time $1$ flow of (the Hamiltonian vector field associated to) the Hamiltonian $rh_t$.

We define the time-dependent Hamiltonian
\[ \ol{h}_t = -h_{1-t} \colon \bC \longrightarrow \bR {\rm ,} \]
and we define
\[ H_t \colon \bC^2 \longrightarrow \bR \colon (z,w) \longmapsto h_t(z) + h_t(w) {\rm ,}\]
and
\[ \ol{H}_t \colon \bC^2 \longrightarrow \bR \colon (z,w) \longmapsto \ol{h}_t(z) + \ol{h}_t(w) {\rm .}\]
We write $\Phi^{r} \colon \bC^2 \longrightarrow \bC^2$ for the time $1$ flow of $rH_t$ so that we have $\Phi^r(\gamma_0 \times \gamma_0) = \gamma_r \times \gamma_r$, and $(\Phi^r)^{-1}$ is the time $1$ flow of $\ol{H}_t$.

We further let
\[ G^\theta_t \colon \bC^2 \longrightarrow \bR \]
be a time-dependent Hamiltonian, identically $0$ for $t$ near $0$ and $1$, that has the form of an area $1$ bump function in $t$ multiplied by the Hamiltonian $\frac\theta4 \vert z - w \vert^2$ for $(z,w) \in \bC^2$.  The inverse of the time $1$ flow of $G^\theta_t$ is rotation $R_\theta\colon \bC^2 \longrightarrow \bC^2$ by angle $\theta$ around the diagonal $\Delta_\bC$.

For $0 \leq r \leq 1$, we are going to consider the concatenation of the Hamiltonians $r\ol{H}_t$, $G_t^\theta$, and $rH_t$, which shall give a time-dependent Hamiltonian \[ F^{r,\theta}_t \colon \bC^2 \longrightarrow \bR\]
for times $0 \leq t \leq 3$.  Writing $\Psi^{r,\theta} \colon \bC^2 \longrightarrow \bC^2$ for the time $3$ flow of $F^{r,\theta}_t$, the pair of Lagrangians
\[(\gamma_0 \times \gamma_0, (\Psi^{r,\theta})^{-1}(\gamma_0 \times \gamma_0)) = (\gamma_0 \times \gamma_0, (\Phi^r)^{-1} \circ R_\theta \circ \Phi^r (\gamma_0 \times \gamma_0))\]
is symplectomorphic to
\[ (\Phi^r(\gamma_0 \times \gamma_0),R_\theta \circ \Phi^r (\gamma_0 \times \gamma_0) )  =  (\gamma_r \times \gamma_r, R_\theta (\gamma_r \times \gamma_r)) {\rm .} \]

It is this last path of pairs of Lagrangians in which we are mainly interested.  In particular we would like to bound the filtered degree of the chain homotopy equivalence
\[ \JFC(\gamma_0, \theta) \longrightarrow \JFC(\gamma_1, \theta) \]
induced by Floer's bifurcation argument.  This bound will be formulated in terms of the Hofer norm $\vert h_t \vert$, and thus in terms of the Hofer distance $d(\gamma_0, \gamma_1)$.

We first give the argument in the case in which there is no Hamiltonian perturbation of the path of Lagrangian pairs.

\begin{lem}
	\label{lem:hoferdistancebound1}
	Suppose that, following the discussion above, for $0 \leq r \leq 1$,
	\[ p_r \in (\gamma_r \times \gamma_r) \cap R_\theta(\gamma_r \times \gamma_r) \]
	gives a path of transverse intersection points.  Then the actions satisfy
	\[ \vert \cA(p_1) - \cA(p_0) \vert \leq 4 \vert h_t \vert {\rm .} \]
\end{lem}

The Hofer norm is tailored to provide action bounds of the kind appearing in \Cref{lem:hoferdistancebound1}.
For instance, it is also behind the proof of the continuity of the spectral invariants in appearing in \Cref{thm:basic_properties} (compare \cite[Theorem 3 and Section 3.2]{leczap}).
The following proof follows a standard line of argument.

\begin{proof}
	Suppose that for $0 \leq r \leq 1$, $p_r$ corresponds the trajectory
	\[ \tau_r \colon [0,3] \longrightarrow \bC^2 \]
	of the Hamiltonian vector field associated to $F^{r,\theta}_t$.
	
	We have
	\[
		\cA(p_1) - \cA(p_0) = \int_0^3 F^{1,\theta}_t \circ \tau_1(t)dt - \int_{D_1} \omega  - \int_0^3 F^{0,\theta}_t \circ \tau_0(t) d t  + \int_{D_0} \omega
	\]
	where $D_0$ and $D_1$ are {\em preferred} cappings of $\tau_0$ and $\tau_1$, i.e. cappings disjoint from $\Delta_\bC$ (see \cite[Lemma 2.5]{greenelobb3}).
	Since $\int_{D_1} \omega$ depends only on the homotopy class of $D_1$, we take $D_1$ to differ from $D_0$ by the image of
	\[ \tau \colon [0,3] \times [0,1] \longrightarrow \bC^2 \colon (t,r) \longmapsto \tau_r(t) {\rm .} \]
	Thus,
	\[
		\cA(p_1) - \cA(p_0) = \int_0^3 F^{1,\theta}_t \circ \tau_1(t) - F^{0,\theta}_t \circ \tau_0(t) d t - \int_{[0,3] \times [0,3]} \tau^*(\omega).
	\]
	
	The last integrand is 
	
		\begin{align*}
		\tau^*(\omega)\left(\frac{\partial}{\partial r},\frac{\partial}{\partial t}\right) &= \omega \left( \frac{\partial \tau_r(t)}{\partial t} , \frac{\partial \tau_r(t)}{\partial r} \right) & \text{definition of pullback} \\
		&= \omega \left( X_{F^{r,\theta}_t} ,  \frac{\partial \tau_r(t)}{\partial r} \right) & \text{definition of Hamiltonian trajectory} \\
		&=  d(F^{r,\theta}_t)\left(\frac{\partial \tau_r(t)}{\partial r}\right) & \text{definition of Hamiltonian vector field} \\
		&=  \left. \frac{\partial (F^{s,\theta}_t) \circ \tau_r(t)}{\partial r}\right\vert_{s=r} & \text{definition of differential} \\
		&=  \left.\frac{\partial (F^{r,\theta}_t \circ \tau_r(t))}{\partial r} - \frac{\partial F^{r,\theta}_t \circ \tau_s(t)}{\partial r}\right\vert_{s=r}  & \text{chain rule.}
		\end{align*}
	Thus,
		 \[ \int_{[0,3] \times [0,3]} \tau^*(\omega)=  \int_0^3 \left( F^{1,\theta}_t \circ \tau_1(t) - F^{0,\theta}_t \circ \tau_0(t)\right) dt - \int_0^3 \int_0^1 \left. \frac{\partial F^{r,\theta}_t \circ \tau_s(t)}{\partial r}\right\vert_{s=r} d r d t \]
	by the Fundamental Theorem of calculus, whence

	\begin{align*}
		\vert \cA(p_1) - \cA(p_0) \vert &= \left\vert \int_0^3 \int_0^1 \left( \left. \frac{\partial F^{r,\theta}_t \circ \tau_s(t)}{\partial r}\right\vert_{s=r} \right) drdt \right\vert = \left\vert \int_{[0,1] \cup [2,3]}  \int_0^1 \left( \left. \frac{\partial F^{r,\theta}_t \circ \tau_s(t)}{\partial r} \right\vert_{s=r} \right) d r d t \right\vert \\
		&\leq  \int_{[0,1] \cup [2,3]} \sup_{0 \leq r \leq 1} \left( \sup F^{r, \theta}_t - \inf F^{r, \theta}_t \right) d t
		\\ &= 2\int_0^1 \left( \sup(H_t) - \inf(H_t) \right) d t
		\leq 4 \vert h_t \vert {\rm ,}
	\end{align*}
	where the second equality follows since $F^{r,\theta}_t$ is constant at a fixed point of $\bC^2$ whenever $1 \leq t \leq 2$.
\end{proof}

With this in hand, we can establish the following result.

\begin{prop}
	\label{prop:hoferdistancebound2}
	Suppose that $\gamma_0$ and $\gamma_1$ are smooth Jordan curves enclosing the same area and $0 < \theta < \pi$.  Then there exists a filtered chain homotopy equivalence
	\[ \JFC(\gamma_0, \theta) \longrightarrow \JFC(\gamma_1, \theta) \]
	which is filtered of degree $4 \d(\gamma_0,\gamma_1)$, and consequently an isomorphism
	\[ \JF(\gamma_0, \theta) \longrightarrow \JF(\gamma_1, \theta) \]
	which is filtered of degree $4\d(\gamma_0,\gamma_1)$.
\end{prop}

The following corollary is immediate.

\begin{cor}
	\label{cor:hoferdistancebound3}
	For $i = 1,2$ and $0 < \theta < \pi$, if $\gamma_0$ and $\gamma_1$ are smooth Jordan curves enclosing the same area, then
	\[ \vert \ell_i(\gamma_0, \theta) - \ell_i(\gamma_1,\theta) \vert \leq 4d(\gamma_0, \gamma_1) {\rm .} \] \qed
\end{cor}

\begin{proof}[Proof of Proposition \ref{prop:hoferdistancebound2}]
	If $h_t \colon \bC \longrightarrow \bR$ is a time-dependent Hamiltonian whose time $1$ flow takes $\gamma_0$ to $\gamma_1$, we can attempt to use Lemma \ref{lem:hoferdistancebound1} to establish the result by using the intermediate Hamiltonians $rh_t$ for $0 \leq r \leq 1$ as before.  Indeed, assuming that away from $\Delta_\bC$ we have that $\gamma_r \times \gamma_r$ is transverse to $R_\theta(\gamma_r \times \gamma_r)$, the argument above shows that Floer's chain complex isomorphism is filtered of degree $4 \vert h_t \vert$.  If, on the other hand, there are some non-transverse intersections away from $\Delta_\bC$, Floer tells us that we can find a path of arbitrarily small Hamiltonian perturbations so that the bifurcation argument applies, as follows.
	
	Let us replace the path of time-dependent Hamiltonians 
	\[ F^{r, \theta}_t \colon \bC^2 \longrightarrow \bR \]
	for $0 \leq r \leq 1$, $0 \leq t \leq 3$, by the perturbed path 
	\[ \ol{F^{r, \theta}_t} \colon \bC^2 \longrightarrow \bR \]
	for $0 \leq r \leq 1$, $0 \leq t \leq 4$ which we obtain from $F^{r, \theta}_t$ by concatenating with the path of small perturbations
	\[ K^r_t \colon \bC^2 \longrightarrow \bR \]
	for $0 \leq r \leq 1$, $0 \leq t \leq 1$ (which we assume is identically $0$ near $t=0$ and $t=1$).  Then there exist $0 < r_1 < r_2 < \cdots < r_n < 1$ so that, away from these values of $r$, $\gamma_0 \times \gamma_0$ is transverse away from $\Delta_\bC$ to the image of $\gamma_0 \times \gamma_0$ under the time $4$ flow corresponding to $\ol{F^{r, \theta}_t}$.
	
	For $r_i < r < r_{i+1}$ suppose, then,  that $\ol{\tau}_r(t)$ is a path of trajectories of the vector field $X_{\ol{F^{r, \theta}_t}}$.  Writing $\cA(\ol{\tau}_{r_i})$ and $\cA(\ol{\tau}_{r+1})$ for the limiting actions, the arguments of Lemma \ref{lem:hoferdistancebound1} apply to show that
	\begin{align*}
		\vert \cA(\ol{\tau}_{r_i}) - \cA(\ol{\tau}_{r+1}) \vert & \leq 4(r_{i+1} - r_i)\vert h_t \vert + \left\vert \int_0^1 \int_0^1 \left. \frac{\partial K^r_t \circ \ol{\tau}_s^t}{\partial r}\right\vert_{s=r} drdt \right\vert \\
		&\leq  4(r_{i+1} - r_i)\vert h_t \vert +
		\int_0^1  \sup_{0 \leq r \leq 1}\left(\sup K^r_t - \inf K^r_t\right)  dt
	\end{align*}
	and this final term may be chosen to be arbitrarily small by choice of perturbation $K^r_t$.
	
	So we see that Floer's argument gives bifurcation isomorphisms
	\[ \JF(\gamma_0, \theta) \longrightarrow \JF(\gamma_1, \theta) \]
	of filtered degree arbitrarily close to $\vert h_t \vert$.  The result now follows by choosing a sequence $h^n_t$ whose Hofer norms limit to the Hofer distance $\d(\gamma_0, \gamma_1)$.
\end{proof}

\subsection{Piecewise linear Jordan curves}
\label{subsec:pwlinear}
For technical reasons, we shall find it convenient to extend the definition of the spectral invariants associated to a smooth Jordan curve to piecewise linear Jordan curves.  We further wish to verify that the spectral invariants vary continuously in a (linearly varying) family of piecewise linear Jordan curves.  Let us begin by defining our terms.

\begin{defin}
	\label{defin:pwlinearjordan}
	A \emph{piecewise linear Jordan curve} (or \emph{PL curve}) $Z \subset \bC$ consists of a finite set of cyclically ordered distinct points
	\[ Z = \{ z_1, z_2, \ldots, z_N \} \subset \bC \]
	such that the interior of each interval $\ol{z_i z_{i+1}}$ is disjoint from every other interval.  We shall often abuse both nomenclature and notation by referring to the union of the intervals itself as a PL curve, and also writing $Z$ to refer to the union of the intervals.
\end{defin}

We permit the possibility that $z_i$ lies on the interval $\ol{z_{i-1} z_{i+1}}$.

\begin{defin}
	\label{defin:pwlinearisotopy}
	Suppose that for $t \in I \subset \bR$ where $I$ is an interval, the sets
	\[ Z(t) = \{ z_1(t), z_2(t), \ldots, z_N(t) \} \subset \bC \]
	are PL curves.  We say that $Z(t)$ is an \emph{isotopy} if each path $z_i(t)$ is piecewise smooth.  We say that $Z(t)$ is a \emph{piecewise linear isotopy} (or \emph{PL isotopy}) if each $z_i(t)$ is a piecewise linear path in which the derivative $\frac{d z_i (t)}{dt}$ is piecewise constant.
\end{defin}

In order to define a spectral invariant for a PL curve, our procedure will be to approximate it by smooth curves and take a limit of their spectral invariants.  We shall be somewhat careful about the kind of convergence considered in order to ensure that the resulting limit is independent of choices.

\begin{defin}
	\label{defin:admissiblePLapprox1}
	Suppose that $z,w \in \bC$ are distinct points and $\epsilon > 0$.  We write $B_\epsilon(z,w) \subset \bC$ for the solid closed rectangle of width $2 \epsilon$ and length $2\epsilon + \vert z - w \vert$, which is centred at $(z + w)/2$ and whose long edges are parallel to the interval $\ol{zw}$.  We shall think of $B_\epsilon(z,w)$ as foliated by the intervals of length $2\epsilon$ that are parallel to the short edges of the rectangle.
\end{defin}

\begin{defin}
	\label{defin:admissiblePLapprox1}
	Suppose that $Z = \{ z_1, \ldots, z_N \} \subset \bC$ is a PL curve.  We say that the smooth Jordan curve $\gamma$ is $(Z,\epsilon)$-\emph{admissable} if there exists some $\epsilon > 0$ such that
	\[ \gamma \subset \bigcup_{i=1}^N B^\circ_\epsilon(z_i, z_{i+1}) \]
	(where the superscript $\circ$ indicates the interior),
	$\gamma$ is transverse to the foliations of each $B_\epsilon(z_i,z_{i+1})$, and
	\[ B_\epsilon(z_i, z_{i+1}) \cap B_\epsilon(z_j, z_{j+1}) = \emptyset \]
	whenever $\{i, i+1\} \cap \{j, j+1\} = \emptyset$, indices (mod $n$).
	We further say that $\gamma$ is $Z$-admissible if $\gamma$ is $(Z,\epsilon)$-admissable for some $\epsilon > 0$.
\end{defin}

We now make a speculative definition that we shall have to check is independent of choices.

\begin{defin}
	\label{defin:spectralforPLcurve}
	If $Z$ is a PL curve and $\gamma_n \subset \bC$ is a $Z$-admissable sequence of smooth curves, each enclosing the same area as $Z$, that converge to $Z$ in the Hausdorff metric $\gamma_n \rightarrow Z$, then we define
	\[ \ell_i(Z, \theta) = \lim_{n \rightarrow \infty} \ell_i(\gamma_n, \theta) \]
	for $i=1,2$ and $0 < \theta < \pi$.
\end{defin}

To check that $\ell_i(Z, \theta)$ is well-defined, we establish the following lemma.

\begin{lem}
	\label{lem:spectralforPLcurve}
	Suppose that $Z$ is a PL curve and $\gamma_n$ and $\gamma'_n$ are two sequences of $Z$-admissable Jordan curves converging to $Z$.  Then 
	\[ \ell_i(\gamma_m, \theta) - \ell_i(\gamma'_n, \theta) \rightarrow 0 \]
	as $m,n \rightarrow \infty$.
\end{lem}

\noindent By taking $\gamma_n = \gamma'_n$ Lemma \ref{lem:spectralforPLcurve} establishes that the limit of Definition \ref{defin:spectralforPLcurve} exists, whereas taking $\gamma_n$, $\gamma'_n$ to be possibly different verifies that the limit is independent of the choice of converging sequence of smooth Jordan curves.

\begin{proof}
	By passing to the tails of the sequences if necessary, we choose some $\epsilon > 0$ so that $\gamma_n$ and $\gamma_n'$ are $(Z,\epsilon)$-admissable for all $n$.
	Fix a choice of $m,n$.
		
	We apply a small Hamiltonian perturbation $h_{m,n} \colon \bC \longrightarrow \bR$ to $\gamma'_n$ so that the time $1$ flow of $\gamma'_n$ under $X_{h_{m,n}}$ is transverse to $\gamma_m$.  Choosing $h_{m,n}$ small enough guarantees that the perturbation of $\gamma'_n$ will remain $(Z,\epsilon)$-admissable.  Hence the perturbation of $\gamma_n'$ is graphical over $\gamma_m$ in the sense that their points of intersection occur in the same cyclic orders when travelling around either curve.
	
	To get a clearer picture, we apply a symplectomorphism to $\bC$ that takes $\gamma_m$ to the unit circle $S^1$, and the perturbation of $\gamma'_n$ to a transversal section $s \colon S^1 \rightarrow A$ of an annular Weinstein neighborhood $A$ of $S^1$.  That we can do this is essentially a result of Karlsson \cite{karlsson}, who treats the case of symplectic isotopies of a self-transverse images of the circle in $\bC$, although her analysis applies equally well to symplectic isotopies of self-transverse connected images of a disjoint union of circles.
	
	The section $s$ is a Hamiltonian perturbation of the $0$-section $S^1$.
	Write $s = df$ for a Morse function $f : S^1 \to \bR$.
	Thus, $s$ is the time 1 flow of $S^1$ under the Hamiltonian $f \circ \pi : T^* S^1 \to \bR$.
	Suppose that $f$ attains its minimum value at $\theta_0 \in S^1$.
	Then $f(\theta)-f(\theta_0) = \int_{\theta_0}^\theta df$.
	Thus, $f(\theta)-f(\theta_0)$ is certainly bounded above by the positive area cobounded by $S^1$ and $s$, and this is half the unsigned area cobounded by $s$ and $S^1$.
	It follows that the Hofer norm of $f \circ \pi$ is bounded above by the half the unsigned area between $s$ and $S^1$.
	But this is half the unsigned area between the perturbation of $\gamma'_n$ and $\gamma_m$.  Choosing the perturbations $h_{m,n}$ to limit to $0$ as $m,n \rightarrow \infty$, this unsigned area also tends to 0 as $m,n \rightarrow \infty$ since $\gamma_m, \gamma'_n \rightarrow Z$.
	
	But now Corollary \ref{cor:hoferdistancebound3} applies to tell us that $\ell_i(\gamma_m, \theta) - \ell_i(\gamma'_n, \theta) \rightarrow 0$.
\end{proof}

Hence we have assigned spectral invariants to each PL curve.
Similar arguments may allow one to assign spectral invariants to other Jordan curves satisfying some regularity hypothesis.
However, it is not clear how one might sensibly assign a spectral invariant to a Jordan curve with no regularity hypothesis.

We now verify that the spectral invariants of PL curves vary continuously in families of PL curves (see Definition \ref{defin:pwlinearisotopy}).

\begin{prop}
	\label{prop:pwlinearisotopyspectral}
	Suppose that for $0 < t < 1$,
	\[ Z(t) = \{ z_1(t), \ldots, z_N(t) \} \]
	is an isotopy of PL curves.
	Then for $i=1,2$ and $0 < \theta < \pi$ we have that
	\[ \ell_i( Z(\cdot), \theta ) \colon (0,1) \longrightarrow \bR \colon t \longmapsto \ell_i(Z(t), \theta) \]
	is continuous.
\end{prop}

\begin{proof}
	First note that the spectral invariants of a $PL$ curve vary continuously with rescaling, since the same is true for smooth curves.  Therefore it is enough to prove the result under the assumption that each $Z(t)$ bounds the same area $A$ for $0 < t < 1$.
	
	Suppose that $a \in (0,1)$.  Then there exist $\delta,\epsilon > 0$ such that whenever $a - \delta < b < a + \delta$ there exist smooth approximations $\gamma_n \rightarrow Z(a)$ and $\gamma'_n \rightarrow Z(b)$ such that $\gamma_n$ and $\gamma'_n$ are both $(Z(a),\epsilon)$ and $(Z(b),\epsilon)$ admissable for each $n$, and so that $\gamma_n$ and $\gamma'_n$ both bound area $A$ for each $n$.
	
	Now we repeat the idea of the proof of Proposition \ref{prop:hoferdistancebound2}.  By perturbing $\gamma'_n$ by a small Hamiltonian $h_n$, we make $\gamma'_n$ both transverse to and graphical over $\gamma_n$.  Then $\vert \ell_i(\gamma_n, \theta) - \ell_i(\gamma'_n, \theta) \vert$ is bounded above by $4d(\gamma_n, \gamma'_n)$ by Corollary \ref{cor:hoferdistancebound3}, and $d(\gamma_n, \gamma'_n)$ is bounded above by half the unsigned area between $\gamma_n$ and $\gamma'_n$.
	
	Taking the perturbations $h_n \rightarrow 0$ as $n \rightarrow \infty$, we see that the limit of this unsigned area as $n \rightarrow \infty$ is the unsigned area between $Z(a)$ and $Z(b)$.  On the other hand,
	\[ \vert \ell_i(\gamma_n, \theta) - \ell_i(\gamma'_n, \theta) \vert \rightarrow \vert \ell_i(Z(a), \theta) - \ell_i(Z(b), \theta) \vert \,\, {\rm as} \,\, n \rightarrow \infty {\rm .} \]
	
	Thus we see that
	\[ \vert \ell_i(Z(a), \theta) - \ell_i(Z(b), \theta) \vert \]
	is bounded above by half the unsigned area between $Z(a)$ and $Z(b)$, and we are done.
\end{proof}
\subsection{Elegant inscriptions.}
\label{subsec:grace}

	\begin{figure}
		\labellist
		\endlabellist
		\centering
		\includegraphics[scale=0.08]{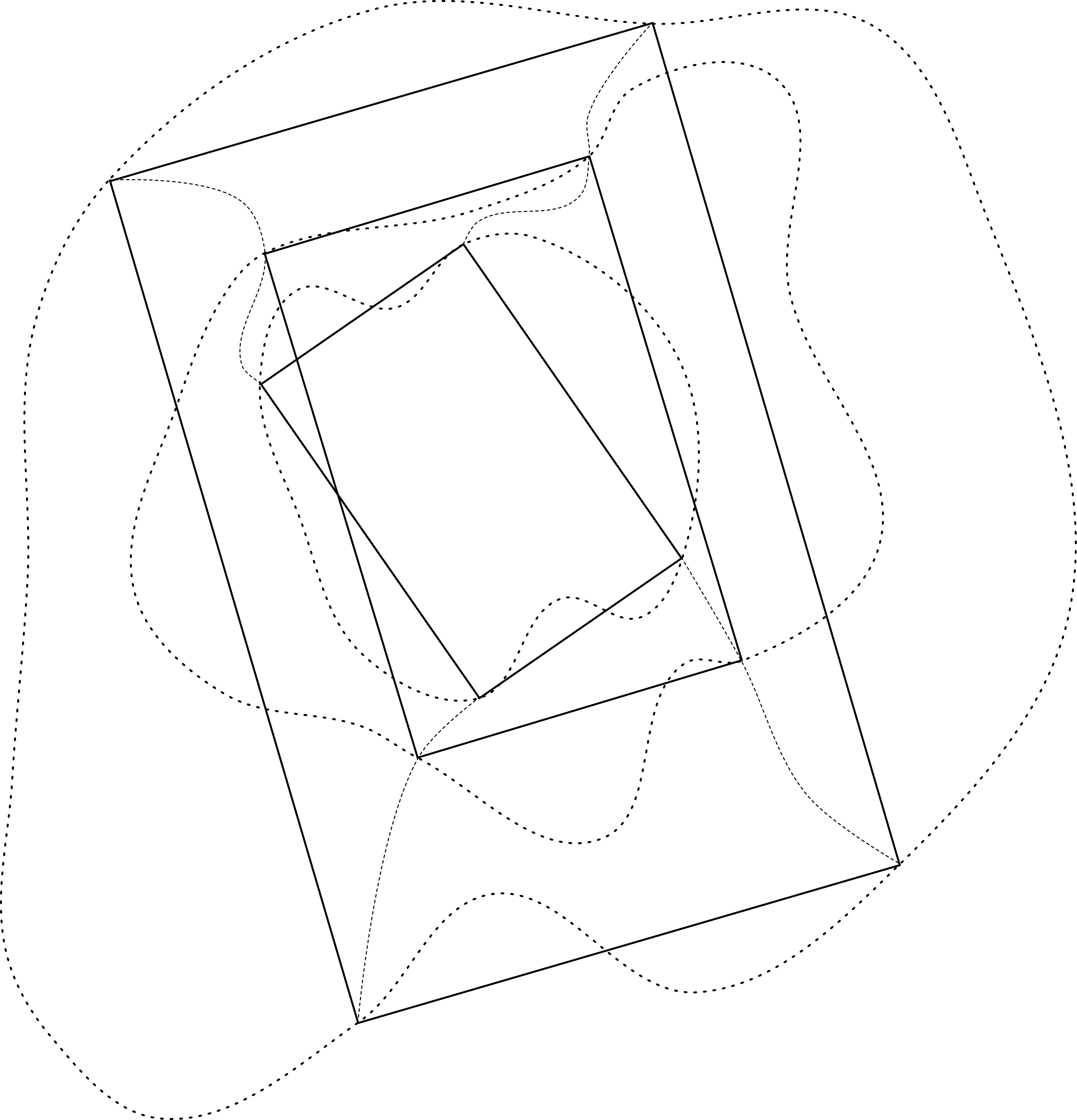}
		\caption{We show an interval's worth of nested Jordan curves $\gamma_r$ for $0 \leq r \leq 1$, together with a path of elegant rectangles $Q_r$ inscribed in $\gamma_r$.}
		\label{fig:path_of_gracefuls}
	\end{figure}

	\begin{figure}
		\labellist
		\pinlabel {$+$} at 900 800
		\pinlabel {$=$} at 2300 800
		\pinlabel {$=$} at 4000 800
		\endlabellist
		\centering
		\includegraphics[scale=0.06]{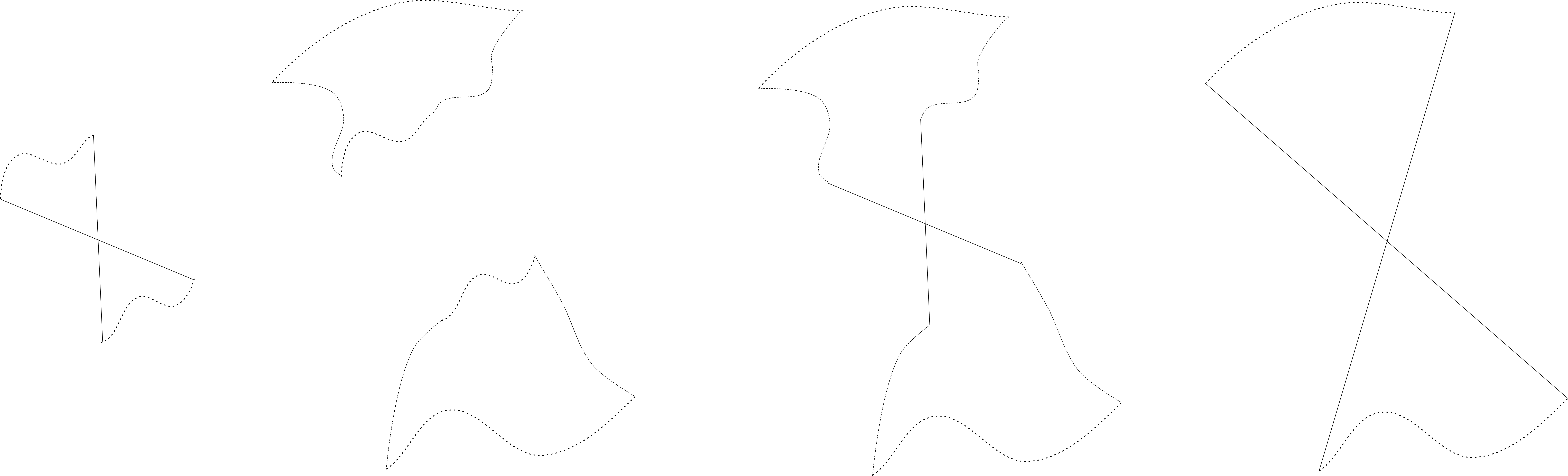}
		\caption{We consider the actions $\cA(Q_0)$ and $\cA(Q_1)$ of rectangles from Figure \ref{fig:path_of_gracefuls}, and give a formula which shows that $0 \leq \cA(Q_1) - \cA(Q_0) \leq \area(\gamma_1) - \area(\gamma_0)$.  The left-most term is the action $\cA(\gamma_0)$, the second term is an area lying between $0$ and $\area(\gamma_1) - \area(\gamma_0)$, the third term the result of adding the two areas on the left, and the third term is the action $\area(\gamma_1)$.}
		\label{fig:variation_action}
	\end{figure}

We consider how the action changes along a path of elegant rectangles.  We note that the action is defined even for rectangles inscribed in piecewise smooth Jordan curves, essentially since contractible loops on piecewise smooth Lagrangians bound zero symplectic area.

\begin{lem}
	\label{lem:addingabitoficecream}
	Suppose that $\gamma_r$ for $0 \leq r \leq 1$ is a nested family of piecewise smooth Jordan curves, and suppose that we also have a family $Q_r \subset \gamma_r$ of elegant $\theta$-rectangles.  Then we have the following
	\[ 0 \leq \cA(Q_1) - \cA(Q_0) \leq \area(\gamma_1) - \area(\gamma_0) {\rm.} \]
\end{lem}

\begin{proof}
	The situation is illustrated in Figure \ref{fig:path_of_gracefuls}, in which is depicted a path of elegant $\theta$-rectangles $Q_r$ (we are imagining $0 < \theta < \pi/2$ in this case) in a nested family of Jordan curves $\gamma_r$.
	
	In Figure \ref{fig:variation_action} we show the calculation that we wish to make to establish the lemma.  The sum on the left is $\cA(Q_0)$ plus some area that lies between $0$ and $\area(\gamma_1) - \area(\gamma_0)$.  The middle term is just the sum of these areas, while the final term is $\cA(Q_1)$.  The non-obvious equality is the second one; we shall give a topological argument.
	
	In Figure \ref{fig:icecreamtopping} we have drawn a piecewise-smooth Jordan curve $\gamma$ made up of two arcs of $\gamma_0$, two arcs of $\gamma_1$, and the four loci traced out by the (vertices of) the rectangles $Q_r \subset \gamma_r$.  Note that the nesting of the curves $\gamma_r$ ensures that $\gamma$ does not self-intersect.
	
	There is a path of intersection points between $\gamma \times \gamma$ and $R_\theta(\gamma \times \gamma)$ corresponding to the path of inscribed rectangles $Q_r \subset \gamma$.  Hence each of these rectangles, when considered as an inscribed rectangle of $\gamma$, has the same action, and one sees from the figure that this is the action of $Q_0$ considered as an inscribed rectangle of $\gamma_0$.  On the other hand, the action of $Q_1$ considered as an inscribed rectangle of $\gamma_1$ differs from the action of $Q_1$ considered as an inscribed rectangle of $\gamma$ (and so differs from the action of $Q_0$ considered as an inscribed rectangle of $\gamma_0$) exactly by the addition of the second term of Figure \ref{fig:variation_action}.  Thus we have verified that the sum on the left of Figure \ref{fig:variation_action} agrees with the right hand side, and so we are done.
\end{proof}

	\begin{figure}
	\labellist
	\endlabellist
	\centering
	\includegraphics[scale=0.08]{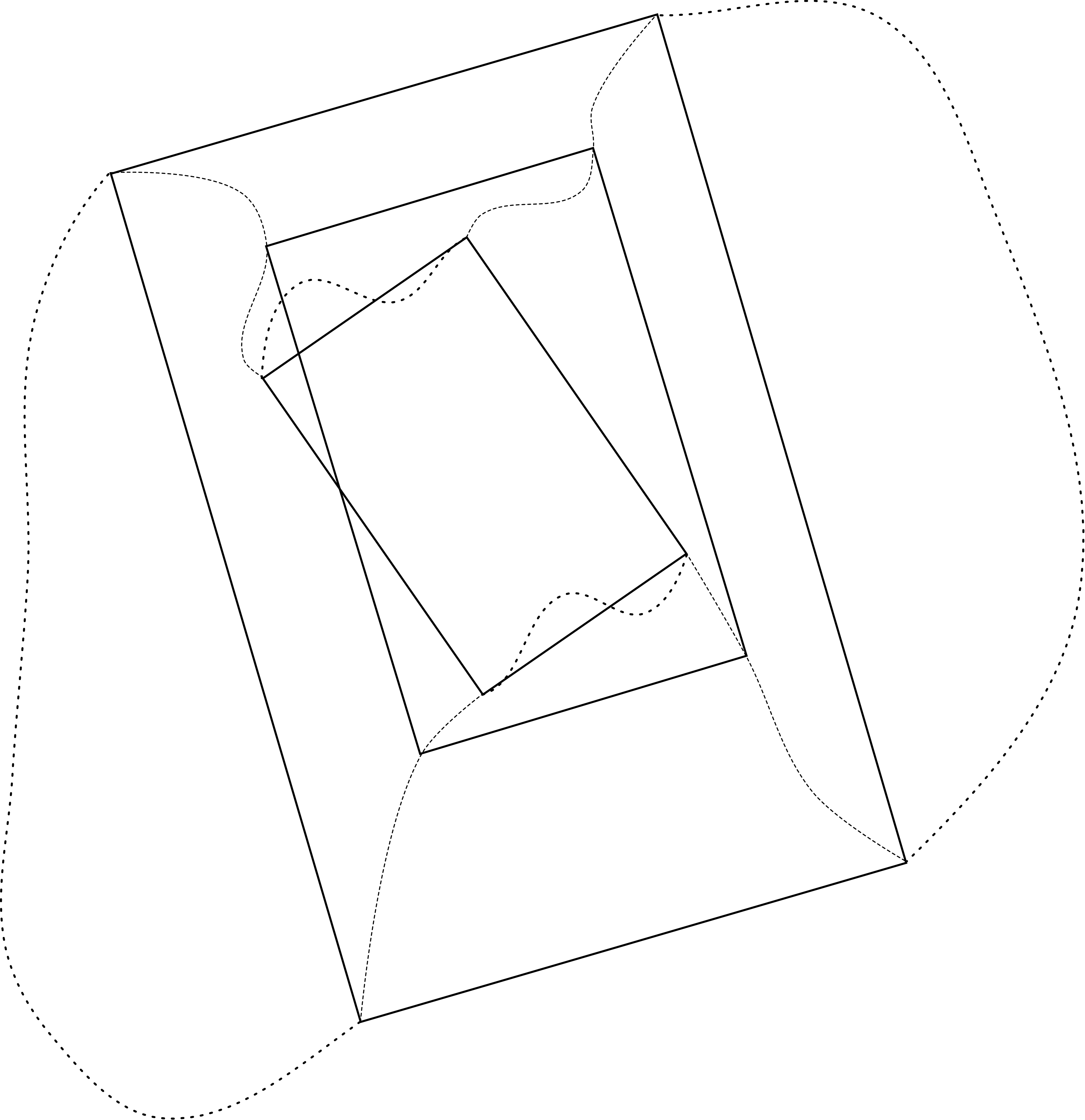}
	\caption{We give a path of rectangles $Q_r$ in a piecewise-smooth Jordan curve $\gamma$.  The action is constant along the path.}
	\label{fig:icecreamtopping}
	\end{figure}

\section{Lipschitz graphs and the proof of Theorem~\ref{thm:thetarect}.}
\label{sec:lipschitz_et_cetera}

Suppose that $0 < \theta \leq \pi/2$ and
\[ f,g \colon [a,b] \longrightarrow \bR \]
are Lipschitz continuous of Lipschitz constant less than $\tan(\frac{\pi + \theta}{4})$, with $f(a)=g(a)$, $f(b)=g(b)$,
and $f(p) > g(p)$ for all $a < p < b$.  We shall write $\gamma(f,g) \subset \bR^2$ for the Jordan curve formed as the union of the graphs of $f$ and $g$,
\[\gamma(f,g) = \Gamma(f) \cup \Gamma(g) 
{\rm .} \]

\begin{lem}
	\label{lem:tao_curves_inscribe_graceful}
	If $Q \subset \gamma(f,g)$ is an inscribed $\theta$-rectangle, then $Q$ is elegant.
\end{lem}

	\begin{figure}
	\labellist
	\pinlabel {$\theta$} at 2900 1750
	\endlabellist
	\centering
	\includegraphics[scale=0.04]{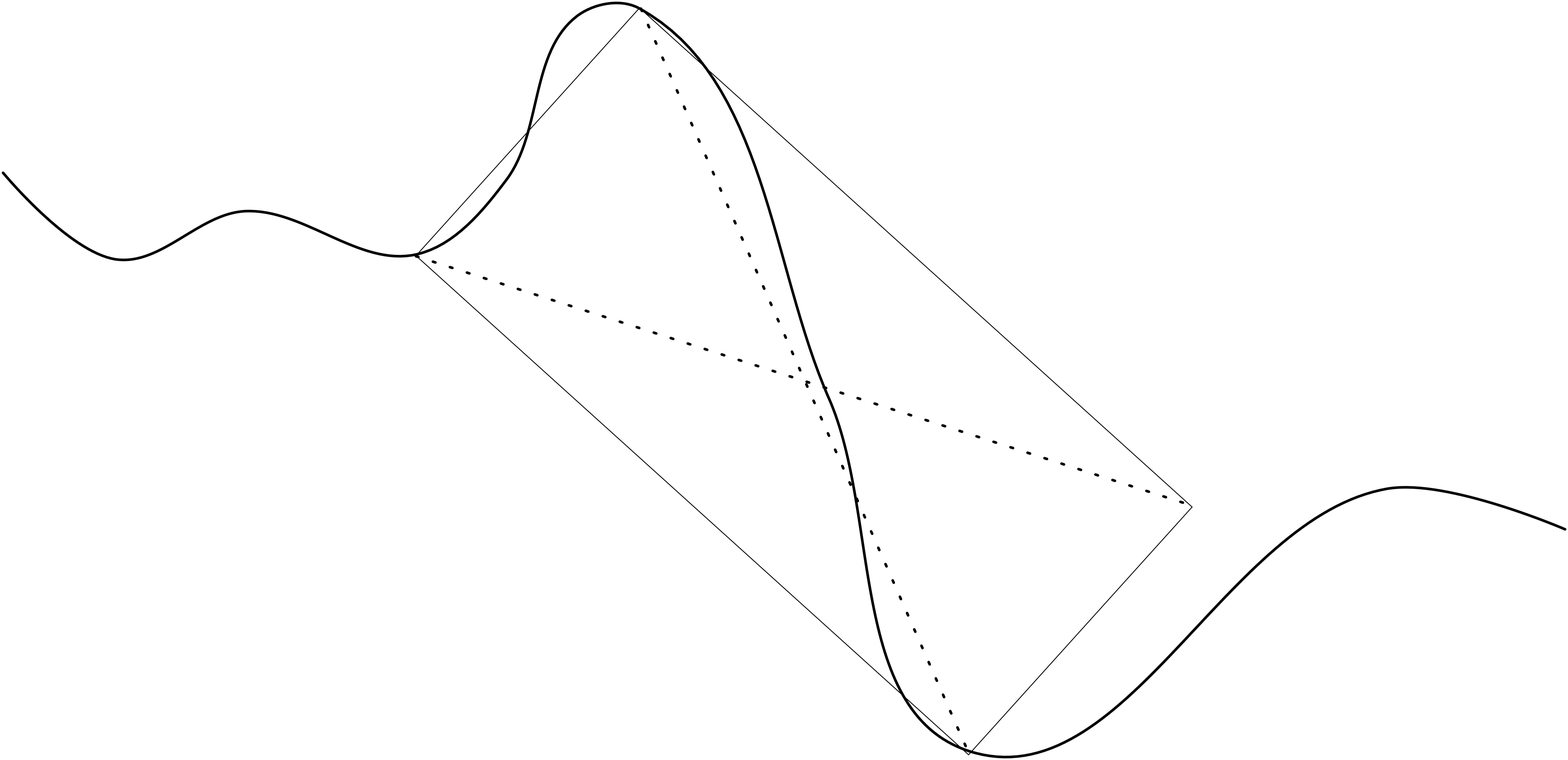}
	\caption{We show a $\theta$-rectangle in which three of the vertices lie on the graph $\Gamma(f)$ of a Lipschitz function $f$.  The maximum of the absolute values of the slopes of the left-hand short edge and the diagonal lying on $\Gamma(f)$ is bounded below by $\tan(\frac{\pi + \theta}{4})$.}
	\label{fig:lipschitz_gets_used}
	\end{figure}

\begin{proof}
	Write the vertices of $Q$ as $v_i = (x_i,y_i)$ for $i = 1,2,3,4$ and where $x_1 \leq x_2 \leq x_3 \leq x_4$.  First note that if any of the inequalities are equality, then $Q$ is necessarily elegant, so let us assume in fact that
	\[ x_1 < x_2 < x_3 < x_4 {\rm ,} \]
	so that the diagonals of the rectangle are $\{ v_1, v_4 \}$ and $\{ v_2, v_3 \}$.
	
	First, we argue that no three consecutive vertices are contained in the same graph $\Gamma(f)$ or $\Gamma(g)$.
	Suppose for a contradiction that this is not the case, and without loss of generality that $v_1, v_2, v_3 \in \Gamma(f)$.  
	The absolute values of the slopes of the lines $v_1 v_2$ and $v_2 v_3$ take the form $\tan(\alpha)$ and $\tan(\pi/2-\alpha+\theta/2)$ for some angle $0 < \alpha < \pi/2$ (see Figure \ref{fig:lipschitz_gets_used}).
	The maximum of these two values is minimized when the arguments agree.
	Thus, at least one of the slopes is at least $\tan(\frac{\pi + \theta}{4})$ in absolute value, contradicting the Lipschitz condition on~$f$.
		
	Second, we argue that $v_2, v_3$ cannot be contained in the same graph $\Gamma(f)$ or $\Gamma(g)$.
	Suppose for a contradiction that they are.
	Without loss of generality, suppose $v_2,v_3 \in \Gamma(f)$.
	By the previous paragraph, we then have $v_1,v_4 \in \Gamma(g)$.
	Assume without loss of generality that $y_1 < y_2$ (equivalently $y_4 > y_3$).  Then, the maximum among the absolute values of the slopes of the lines $v_2v_3$, $v_3v_4$ must be at least $\tan(\frac{\pi + \theta}{4})$, again contradicting the Lipschitz condition on $f$ since $f(x_4) > y_4$.
	
	It follows that $v_2$ and $v_3$ belong on different graphs. 
	Without loss of generality, $v_2 \in \Gamma(f)$ and $v_3 \in \Gamma(g)$.
	Walking clockwise around $\gamma(f,g)$ one passes through $v_1$, then $v_2$, then $v_4$, then $v_3$.
	Hence $Q$ is gracefully inscribed in $\gamma(f,g)$; but a graceful inscription in $\gamma(f,g)$ is necessarily elegant, and this concludes the proof.
	
\end{proof}

With this in hand, and following our work in the previous section, we are ready to establish Proposition \ref{prop:varying_through_graceful}.

\begin{proof}[Proof of Proposition \ref{prop:varying_through_graceful}]  Suppose that $Z(t) \subset \bC$ for $0 \leq t \leq 1$ is a PL isotopy of PL Jordan curves $Z(t) = \{ z_1(t), \ldots, z_n(t) \}$ such that if $0 \leq s < t \leq 1$ then $Z(s)$ is contained inside the bounded complementary region to $Z(t)$, and further suppose that for each $0 \leq t \leq 1$, all $\theta$-rectangles inscribed in $Z(t)$ are elegant.
	
	Firstly we note that since $Z(t)$ is piecewise linear, there are finitely many components of
	\[ Z(t) \times Z(t) \cap R_\theta(Z(t) \times Z(t)) {\rm .} \]
	The diagonal $\Delta_{Z(t)}$ is a 1-dimensional component, although there may be other 1-dimensional components as well (such as the case in which $Z(t)$ is a rectangle).
	Furthermore, since $Z(t)$ is a PL isotopy of PL curves, there are finitely many times
	\[ 0 = t_1 < t_2 < \cdots < t_M = 1 \]
	such that the map
	\[Z(t) \times Z(t) \cap R_\theta(Z(t) \times Z(t)) \longmapsto t \]
	is a product fibration over $(t_j, t_{j+1})$ for $1 \leq j \leq M-1$.\footnote{It is for this reason that have taken the trouble to verify the good behavior of the spectral invariants for isotopies of PL graphs.}
	
	Let us write $\cA(Z(t),\theta) \subset \bR$ for the set of the actions of all $\theta$-rectangles inscribed in $Z(t)$.
	Observe that $\cA$ is constant on each component of $Z(t) \times Z(t) \cap R_\theta(Z(t) \times Z(t))$, even those of positive dimension.
	Then we see that, for $1 \leq j \leq M-1$, 
	\[ A := \bigcup_{t_j < t < t_{j+1}} \{ \{t\} \times \cA(Z(t),\theta) \} \subset (t_j, t_{j+1}) \times \bR \]
	is the union of the graphs of finitely many functions, say of
	\[ a_k \colon (t_j, t_{j+1}) \longrightarrow \bR \]
	for $1 \leq k \leq P_j$.
	Furthermore we have by Proposition \ref{lem:addingabitoficecream}
	\[ 0 \leq a_k(t) - a_k(s) \leq \area(Z(t)) - \area(Z(s)) \]
	for all $t_j < s < t < t_{j+1}$ and $1 \leq k \leq P_j$.
	
	Now by spectrality and continuity (Proposition \ref{prop:pwlinearisotopyspectral}), the graph of
	\[ \ell_i(Z(\cdot), \theta) \colon (t_j, t_{j+1}) \longrightarrow \bR \colon t \longmapsto \ell_i(Z(t), \theta) \]
	is a subset of $A$.  Therefore we have
	\[ 0 \leq \ell_i(Z(t),\theta) - \ell_i(Z(s),\theta) \leq \area(Z(t)) - \area(Z(s)) \]
	for all $t_j < s < t < t_{j+1}$.  So by continuity of $\ell_i(Z(\cdot), \theta)$ we have
	\[ 0 \leq \ell_i(Z(t_{j+1}),\theta) - \ell_i(Z(t_j),\theta) \leq \area(Z(t_{j+1})) - \area(Z(t_j)) \]
	for $1 \leq j \leq M-1$.  Summing from $j=1$ to $j=M-1$ gives
	\[ 0 \leq \ell_i(Z(1),\theta) - \ell_i(Z(0), \theta) \leq \area(Z(1)) - \area(Z(0)) \]
	as required.
\end{proof}

We now build up an implication chain to establish Theorem \ref{thm:thetarect}.  Let us suppose then that $f,g \colon [0,1] \longrightarrow \bR$ satisfy the conditions at the start of this section.

\begin{prop}
	\label{prop:canapproxtaocurvebyPLcurves}
	There exists a $\delta > 0$ and a sequence of PL curves $C_n$ of uniformly bounded length, such that $C_n \rightarrow \gamma(f,g)$ and
	\[ \delta < \ell_i(C_n,\theta) < \area(C_n) - \delta {\rm .} \]
\end{prop}

Assuming this proposition we have
\begin{proof}[Proof of Theorem \ref{thm:thetarect}]
	The `no shrinkout' Lemma 5.1 of \cite{greenelobb3} (whose proof extends without alteration to the PL case) established the existence of a sequence of $\theta$-rectangles $Q_n \subset C_n$, each of diagonal length bounded below away from 0.  Compactness gives a subsequence that converges to a (non-degenerate) inscribed rectangle of $\gamma(f,g)$.
\end{proof}

Therefore it remains to establish Proposition \ref{prop:canapproxtaocurvebyPLcurves}.

\begin{lem}
	\label{lem:approximatingthetaocurvebysequencesofPLisotopies}
	For each $n \geq 1$ and $0 \leq t \leq 1$, there exists a PL isotopy of PL curves $\gamma(F_n(t),G_n(t))$ such that
	\begin{itemize}
		\item $\gamma(F_n(0), G_n(0)) = S$ where $S$ is a square Jordan curve,
		\item $\gamma(F_n(t),G_n(t))$ is contained in the bounded complementary region to $\gamma(F_n(s),G_n(s))$ for all $0 \leq t < s \leq 1$,
		\item $F_n(t), G_n(t) \colon [a_n(t), b_n(t)] \longrightarrow \bR$ are PL functions of Lipschitz constant less than $\tan(\frac{\pi + \theta}{4})$,
		\item $\gamma(F_n(1), G_n(1))\rightarrow \gamma(f,g)$ as $n \rightarrow \infty$.
	\end{itemize}
\end{lem}

Assuming Lemma \ref{lem:approximatingthetaocurvebysequencesofPLisotopies}, we have
\begin{proof}[Proof of Proposition \ref{prop:canapproxtaocurvebyPLcurves}]
	The curves $C_n$ of Proposition \ref{prop:canapproxtaocurvebyPLcurves} will be the curves $\gamma(F_n(1),G_n(1))$ of Lemma \ref{lem:approximatingthetaocurvebysequencesofPLisotopies}.  First note that these curves are certainly of uniformly bounded length by the Lipschitz property.
	
	Next note that by Lemma \ref{lem:tao_curves_inscribe_graceful}, we know that the families $\gamma(F_n(t),G_n(t))$ only inscribe elegant rectangles.  So by Proposition \ref{prop:varying_through_graceful} we have that
	\[ \ell_i(S,\theta) = \ell_i(\gamma(F_n(0), G_n(0)), \theta) \leq  \ell_i(\gamma(F_n(1),G_n(1)), \theta) = \ell_i(C_n,\theta) \leq \area(C_n) - \ell_i(S,\theta) {\rm .} \]
	
	It remains to observe that every $\theta$-rectangle inscribed in a square Jordan curve has positive action, so that we may set $\delta = \ell_i(S,\theta)$.
\end{proof}

\begin{figure}
	\labellist
	\pinlabel {$P^n_{\rm L}$} at 190 -50
	\pinlabel {$P^n_{\rm R}$} at 750 120
	\pinlabel {$S$} at 550 320
	\pinlabel {$v_0$} at 30 150
	\pinlabel {$v_1$} at 60 -40
	\pinlabel {$v_n$} at 250 200
	\endlabellist
	\centering
	\vspace*{0.5cm}
	\includegraphics[scale=0.2]{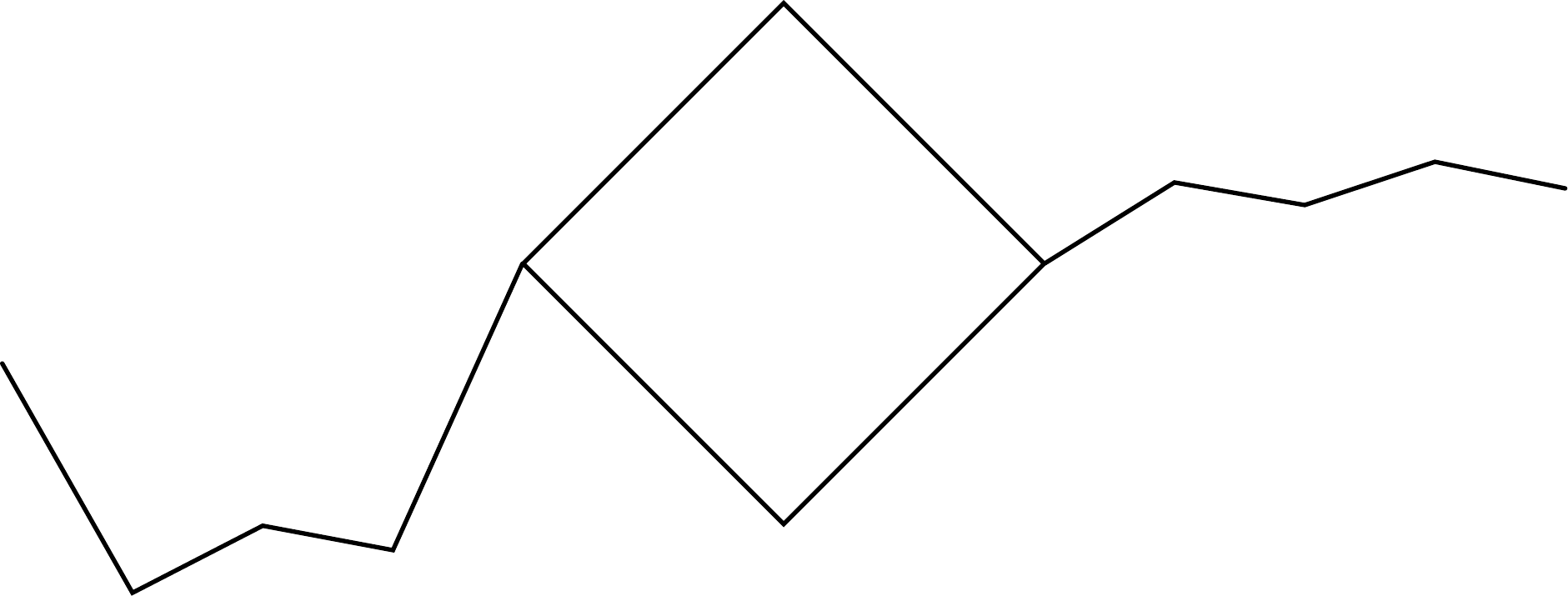}
	\vspace*{0.5cm}
	\caption{We illustrate $P^n_{\rm L} \cup S \cup P^n_{\rm R}$.  The central square is $S$ and the two tendrils running off to either side are the paths $P^n_{\rm L}$ and $P^n_{\rm R}$.  This union can be thought of as something like a \emph{core} of the Jordan curve $\gamma(f,g)$.  We have labelled some of the vertices on the left hand path $P^n_{\rm L}$.}
	\label{fig:PLtendrillycurve}
\end{figure}
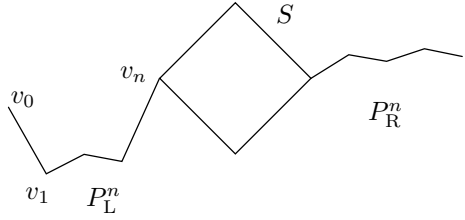

The only result left to establish is Lemma \ref{lem:approximatingthetaocurvebysequencesofPLisotopies}.

\begin{proof}[Proof of Lemma \ref{lem:approximatingthetaocurvebysequencesofPLisotopies}]
	We shall let $S$ be a small square, centred at $\left( 1/2, \frac{f(1/2) + g(1/2)}{2} \right)$, so that $S$ is entirely contained in the interior of the region bounded by $\gamma(f,g)$, and so that the diagonals of $S$ are vertical and horizontal.
	
	Then we have that the left vertex and the right vertex of $S$ have coordinates
	\[ \left(r, s_{\rm L}f(r) + (1-s_{\rm L})g(r)\right) \,\,\, {\rm and} \,\,\, \left( 1-r, s_{\rm R}f(1-r) + (1-s_{\rm R})g(1-r) \right) \]
	for some $0 < r < 1/2$, $0 < s_{\rm L}, s_{\rm R} < 1$.
	
	We shall now describe the construction of the PL isotopy $\gamma(F_n(t), G_n(t))$.
	
	First let us consider the union of $S$ with two piecewise linear paths $P^n_{\rm L}$ and $P^n_{\rm R}$.
	
	The path $P^n_{\rm L}$ has $n+1$ vertices $v_0, v_1, \ldots, v_n$ with edges connecting $v_{i-1}$ to $v_{i}$ for each $1 \leq i \leq n$, and
	\[ v_i = \left( r/i , s_{\rm L}f(r/i) + (1-s_{\rm L})g(r/i) \right) {\rm .} \]
	So each vertex of $P^n_{\rm L}$ lies between the graphs of $f$ and $g$ (although its edges may not).  Thinking of $P^n_{\rm L}$ itself as the graph of a function on $[0,r]$, this function satisfies the same Lipschitz constraints as $f$ and $g$.

	The path $P^n_{\rm R}$ is defined similarly, and runs between the right hand vertex of $S$ and the rightmost point of $\gamma(f,g)$.  See Figure \ref{fig:PLtendrillycurve}.

	To avoid overcomplicating the picture, we shall give a description in which for $s < t$, $\gamma(F_n(s), G_N(s))$ is contained in the \emph{closed} bounded region bounded by $\gamma(F_n(t), G_n(t))$, although the reader will readily verify that the small adjustments to the linear isotopy allow one to achieve containment in the interior.
	
	For $0 \leq t \leq 1/2$, we shall define the PL isotopy $\gamma(F_n(t), G_n(t))$ to begin at $t=0$ with $S$, and to end at $t=1/2$ with a PL curve that is the boundary of a small neighbourhood of $P^n_{\rm L} \cup S \cup P^n_{\rm R}$.  This isotopy will happen in $n$ stages -- each stage involving pushing the curve out along the next edge of $P^n_{\rm L}$ and of $P^n_{\rm R}$.
	
	\begin{figure}
		\labellist
		\endlabellist
		\centering
		\vspace*{0.5cm}
		\includegraphics[scale=0.15]{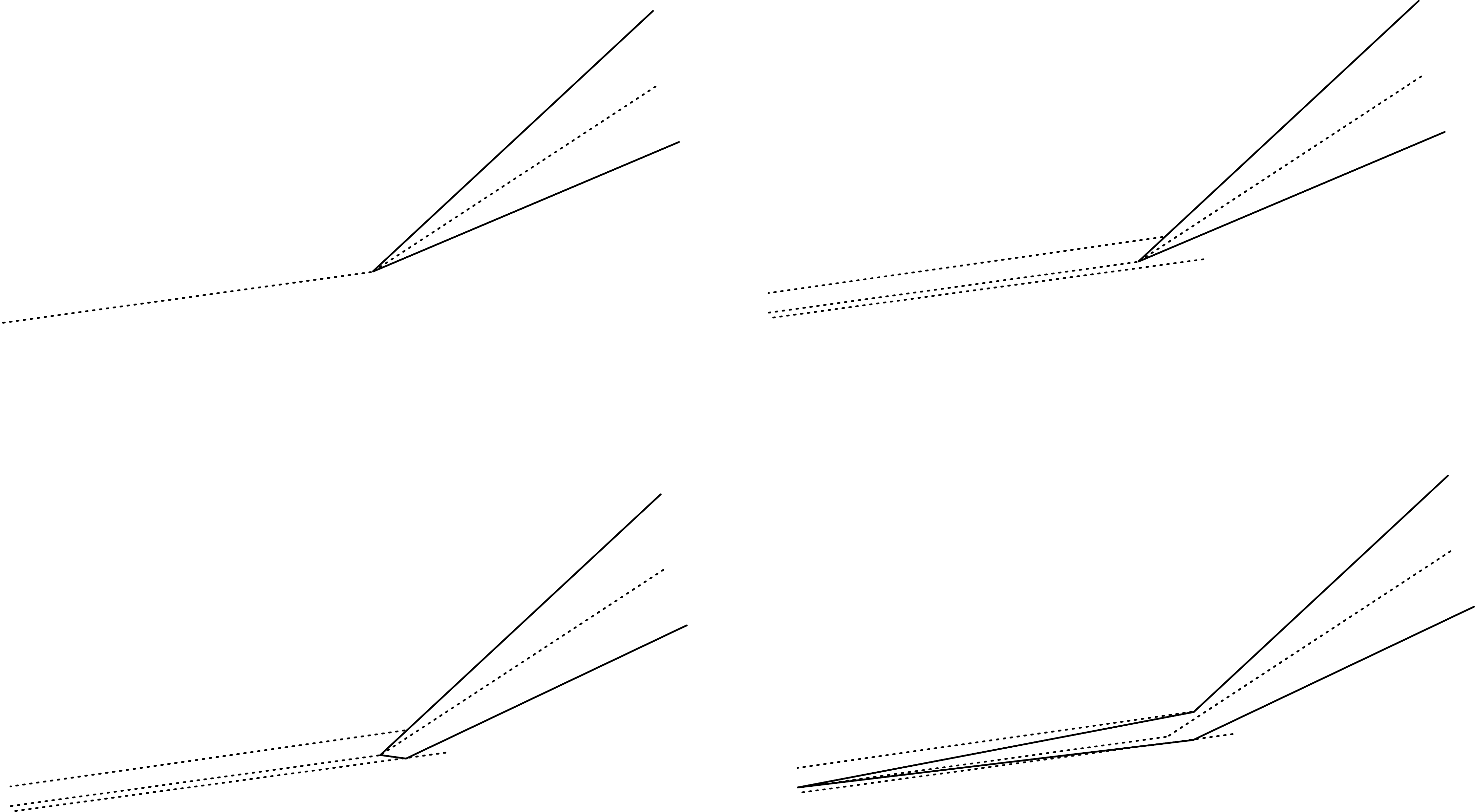}
		\vspace*{0.5cm}
		\caption{We illustrate a stage in pushing out the square along the PL path $P^n_{\rm L}$.  The figure is to be read from left to right and from top to bottom.  The first image shows the result of the previous stage.  The PL curve undergoing the PL isotopy is drawn in a solid line, and two edges of the path $P^n_{\rm L}$ are drawn in dotted lines.  The second image differs from the first image only by the addition of extra guide lines, also dotted, which are parallel to the edge of $P^n_{\rm L}$ which we are next to run over (one may take these guide lines to be equidistant from the parallel edge -- we have chosen not to do so in this picture).  In the third image an internal vertex (in other words subtending angle $\pi$) of the PL curve has moved vertically downwards to meet the lower guideline in the same $x$ coordinate as the curve meets the upper guide line.  Where the curve meets the upper guide line there is another internal vertex.  In the fourth image the leftmost vertex has travelled linearly down the next edge of $P^n_{\rm L}$ to the next vertex of $P^n_{\rm L}$.}
		\label{fig:PLtendrillyisotopystage}
	\end{figure}
	
	We illustrate a stage of the process in Figure \ref{fig:PLtendrillyisotopystage}.  The reader will note that in choosing the guide lines (described in the caption to that figure) very close to the edge along which they run, we can ensure that at each stage $F_n(t)$ and $G_n(t)$ remain curves of the required Lipschitz constant.
	
	We have landed at the curve $\gamma(F_n(1/2), G_n(1/2)) = \Gamma(F_n(1/2)) \cup \Gamma(G_n(1/2))$ in which the graphs $\Gamma(F_n(1/2))$ and $\Gamma(G_n(1/2))$ are each naturally piecewise linear paths of $2n+3$ vertices, some of which may subtend angle $\pi$.  Both to $\Gamma(F_n(1/2))$ and to $\Gamma(G_n(1/2))$ we add $2n$ more vertices subtending angle $\pi$, $n$ each equally spaced along each edge of the square $S$.
	
	Suppose that the $i$th vertex of $\Gamma(F_n(1/2))$ has coordinates $(x_i, F_n(1/2)(x_i))$.  For $1/2 \leq t \leq 1$ we define
	\[ F_n(t) \colon [0,1] \rightarrow \bR\]
	to be a PL function with vertices at
	\[ \left( x_i, 2(t-1/2)f(x_i) + 2(1-t)F_n(1/2)(x_i) \right) {\rm ;} \]
	in other words, with vertices that move linearly upwards with $t$ until they hit the graph $\Gamma(f)$ at time $t=1$.  Similarly we let the vertices of $\Gamma(G_n(1/2))$ move linearly downwards until they hit the graph $\Gamma(g)$.
	
	This completes the construction of the PL isotopy $\gamma(F_n(t), G_n(t))$ for $0 \leq t \leq 1$.
	
	Finally, we observe that $F_n(1) \rightarrow f$ and $G_n(1) \rightarrow g$ as $n \rightarrow \infty$ since the PL approximations are made with arbitrarily fine meshes.
\end{proof}

\bibliographystyle{amsplain}
\bibliography{References./works-cited.bib}
\end{document}